\documentclass[10pt]{amsart}
\usepackage{latexsym, amsmath, amssymb, longtable, booktabs,amscd}

\usepackage[all]{xy}
\UseComputerModernTips

\numberwithin{equation}{section}
\begin{document}

\newcommand\A{\mathbb{A}}
\newcommand\C{\mathbb{C}}
\newcommand\G{\mathbb{G}}
\newcommand\N{\mathbb{N}}
\newcommand\T{\mathbb{T}}
\newcommand\sE{{\mathcal{E}}}
\newcommand\sF{{\mathcal{F}}}
\newcommand\sG{{\mathcal{G}}}
\newcommand\GL{{\mathrm{GL}}}
\newcommand\Ham{\mathbb{H}}
\newcommand\HH{{\mathrm H}}
\newcommand\mM{{\mathrm M}}
\newcommand\Pone{\mathbb{P}}
\newcommand\Qbar{{\bar{\Q}}}
\newcommand\sQ{{\mathcal{Q}}}
\newcommand{\Q}{\mathbb{Q}}
\newcommand{\Z}{\mathbb{Z}}
\newcommand{\Zbar}{{\bar{\Z}}}
\newcommand{\R}{\mathbb{R}}
\newcommand{\F}{\mathbb{F}}
\newcommand\gP{\mathfrak{P}}
\newcommand\Gal{{\mathrm {Gal}}}
\newcommand\Hom{{\mathrm {Hom}}}
\newcommand\soc{{\mathrm {soc} \> }}
\newcommand\Sym{{\mathrm {Sym}}}
\newcommand{\legendre}[2] {\left(\frac{#1}{#2}\right)}
\newcommand\iso{{\> \simeq \>}}
\newcommand\dott{{
         {\begin{tiny}
         \begin{xymatrix}{
                      \\ 
                     \\ 
          \bullet }
        \end{xymatrix} 
        \end{tiny}} }} 
\newcommand\linee{{
        {\begin{tiny}
        \begin{xymatrix}{
         \bullet  \ar@{-}[d] \\ 
         \bullet  \ar@{-}[d] \\
         \bullet} 
       \end{xymatrix}
       \end{tiny}} }}  
\newcommand\diamondd{{
        \begin{tiny}
        \begin{xymatrix}{
         & \bullet  \ar@{-}[dl] \ar@{-}[dr] &  \\ 
         \bullet  \ar@{-}[dr] &   &   \bullet  \ar@{-}[dl] \\
         & \bullet & } 
       \end{xymatrix}
       \end{tiny} }}

\newtheorem{thm}{Theorem}[section]
\newtheorem{theorem}[thm]{Theorem}
\newtheorem{cor}[thm]{Corollary}
\newtheorem{conj}[thm]{Conjecture}
\newtheorem{prop}[thm]{Proposition}
\newtheorem{lemma}[thm]{Lemma}
\theoremstyle{definition}
\newtheorem{definition}[thm]{Definition}
\theoremstyle{remark}
\newtheorem{remark}[thm]{Remark}
\newtheorem{example}[thm]{Example}
\newtheorem{claim}[thm]{Claim}

\newtheorem{lem}[thm]{Lemma}

\theoremstyle{definition}
\newtheorem{dfn}{Definition}

\theoremstyle{remark}

\theoremstyle{remark}
\newtheorem*{fact}{Fact}

\title[]{Reductions of Galois representations for slopes in $(1,2)$}

\author{Shalini Bhattacharya}
\address{School of Mathematics, Tata Institute of Fundamental Research, Homi Bhabha Road, Mumbai-5, India}
\email{shalini@math.tifr.res.in, eghate@math.tifr.res.in}

\author{Eknath Ghate}

\begin{abstract}
  We describe the semisimplification of the mod $p$ reduction of certain crystalline two 
  dimensional 
  local Galois representations of slopes in $(1,2)$ and all weights. The proof uses the mod $p$
  Local Langlands Correspondence for $\mathrm{GL}_2(\Q_p)$. 
  We also give a complete description of the submodules generated by the second 
  highest monomial in the mod $p$ symmetric power representations of 
  ${\mathrm{GL}}_2(\F_p)$.  
\end{abstract}


\maketitle

\section{Introduction}

Let $p$ be an odd prime. 
In this paper we study the reductions of two dimensional crystalline $p$-adic representations of the local
Galois group $G_{\Q_p}$. The answer is known when the weight $k$ is smaller than $2p+1$  \cite{Edixhoven92}, 
\cite{Breuil03a}, \cite{[Br03]} or when the slope is greater than 
$\lfloor \frac{k-2}{p-1} \rfloor$ \cite{Berger-Li-Zhu04}. 
The answer is also known if the slope is small, that is, in the range $(0,1)$ 
\cite{[BG09]}, \cite{Ganguli}, \cite{Buzzard-Gee13}. 
Here we treat the next range of fractional slopes $(1,2)$, for all weights $k \geq 2$.

Let $E$ be a finite extension field of $\Q_p$ and let $v$ be the valuation of $\bar\Q_p$ 
normalized so that $v(p) = 1$. Let $a_p \in E$ with $v(a_p) > 0$ and let $k \geq 2$. 
Let $V_{k,a_p}$ be the irreducible crystalline representation of $G_{\Q_p}$ with Hodge-Tate weights $(0,k-1)$ such that 
$D_\mathrm{cris}(V_{k,a_p}^*) = D_{k,a_p}$, 
where $D_{k,a_p}  = E e_1 \oplus E e_2$ is the filtered $\varphi$-module as defined in \cite[\S 2.3]{Berger11}.
Let $\bar{V}_{k,a_p}$ be the semisimplification of the reduction of ${V}_{k,a_p}$, thought of as a
representation over $\bar\F_p$. 

Let $\omega = \omega_1$ and $\omega_2$ denote the fundamental characters of level 1 and 2 
respectively, and let $\mathrm{ind}(\omega_2^{a})$ denote the  representation of $G_{\Q_p}$ obtained
by inducing the character $\omega_2^a$ from $G_{\Q_{p^2}}$. 
Let $\mathrm{unr}(x)$ be the unramified character of $G_{\Q_p}$ taking (geometric) Frobenius at $p$ to $x \in \bar\F_p^\times$. 
Then, {\it a priori}, $\bar{V}_{k,a_p}$ is isomorphic either to  
$\mathrm{ind}(\omega_2^a) \otimes \mathrm{unr}(\lambda)$ or $\mathrm{unr}(\lambda) \omega^a \oplus \mathrm{unr}(\mu) \omega^b$, for some
$a$, $b \in \Z$ and $\lambda$, $\mu \in \bar{\F}_p^\times$. The former representation is irreducible on $G_{\Q_p}$ when $(p+1)\nmid a$, whereas the latter is reducible on $G_{\Q_p}$.  
The following theorem describes $\bar{V}_{k,a_p}$ when $1 < v(a_p) < 2$. Since the answer is known completely for
weights $k \leq 2p+1$, we shall assume 
that $k \geq 2p+2$.

\begin{theorem}[] 
\label{theorem-intro}
Let $p \geq 3$. Let $1 < v(a_p) < 2$ and $k \geq 2p + 2$. 
Let $r = k-2 \equiv b \mod (p-1)$, with $2 \leq b \leq p$.   
When $b = 3$ and $v(a_p) = \frac{3}{2}$, assume that \begin{equation}\tag{$\star$}\label{hyp} v(a_p^2 - {r-1 \choose  2}(r-2)p^3) = 3.\end{equation}
Then, $\bar{V}_{k,a_p}$ has the following shape on $G_{\Q_p}$:
\begin{eqnarray*}
  b = 2 & \implies &
               \begin{cases} 
                  \mathrm{ind}(\omega_2^{b+1}), & \text{if } p \nmid r(r-1) \\
                  \mathrm{ind}(\omega_2^{b+p}), & \text{if } p \> | \> r(r-1),
               \end{cases} \\
  3 \leq b \leq p-1 & \implies & 
               \begin{cases} 
                  \mathrm{ind}(\omega_2^{b+p}), & \text{if } p \nmid r-b \\
                  \mathrm{ind}(\omega_2^{b+1}), & \text{if } p \> | \> r-b,
               \end{cases} \\
    b = p & \implies &
               \begin{cases} 
                  \mathrm{ind}(\omega_2^{b+p}), & \text{if } p^2 \nmid r-b  \\
                  \mathrm{unr}(\sqrt{-1}) \omega \oplus \mathrm{unr}(-\sqrt{-1}) \omega , & \text{if } p^2 \> | \> r-b.
               \end{cases} 
\end{eqnarray*}
\end{theorem}

Using the theorem, and known results for $2 \leq k \leq 2p+1$, we obtain:

\begin{cor}
  Let $p \geq 3$. If $k \geq 2$ is even and $v(a_p)$ lies in $(1,2)$, then $\bar{V}_{k,a_p}$ is irreducible. 
\end{cor}

It is in fact conjectured \cite[Conj. 4.1.1]{Buzzard-Gee15} that if $k$ is even and $v(a_p)$ is non-integral,
then the reduction $\bar V_{k,a_p}$ is irreducible on $G_{\Q_p}$. This follows for slopes in $(0,1)$ by \cite{[BG09]}. 
Theorem~\ref{theorem-intro} shows that $\bar{V}_{k,a_p}$ can be reducible on $G_{\Q_p}$ for slopes in $(1,2)$ only when 
$b=p$ or $b=3$ (or both).
Since $k$ is clearly odd in these cases, the corollary follows.

Let $\rho_f : \mathrm{Gal}(\Qbar/\Q) \rightarrow \mathrm{GL}_2(E)$ denote the global Galois representation attached to a 
primitive cusp form $f = \sum a_n q^n \in S_k(\Gamma_0(N))$ of (even) weight $k \geq 2$ and level $N$ coprime to $p$. 
It is known that $\rho_f |_{G_{\Q_p}}$ is isomorphic to $V_{k,a_p}$ at least if $a_p^2 \neq 4 p^{k-1}$. This condition is 
always true for slopes in $(1,2)$ unless $k = 4$ and $a_p = \pm 2 p^{\frac{3}{2}}$ and is expected to hold generally,
so we assume it. 
We obtain:

\begin{cor}
  \label{modular}
  Let $p \geq 3$. If the slope  of $f$ at $p$ lies in $(1,2)$, then $\bar\rho_f |_{G_{\Q_{p}}}$ is irreducible.
\end{cor}

\begin{remark} 
Here are several remarks. 
\begin{itemize}
  \item Theorem~\ref{theorem-intro} treats all weights for slopes in the range $(1,2)$, subject to a hypothesis.  
        It builds on  
        \cite[Thm. 2]{[GG15]}, which treated weights less than $p^2-p$. 
  \item The hypothesis \eqref{hyp} in the theorem applies only when 
        $b = 3$ and $v(a_p) = \frac{3}{2}$ and is mild in the sense that it holds whenever 
        the unit $\frac{a_p^2}{p^3}$ and ${r-1 \choose 2}(r-2)$ have distinct reductions in $\bar\F_p$.  
  \item The theorem agrees with all previous results for  
        weights $2<k \leq 2p+1$ described in \cite[Thm. 5.2.1]{Berger11} 
        when specialized to slopes in $(1,2)$. It could therefore be stated for all weights $k > 2$. We note
        that \eqref{hyp} is satisfied for weights $k \leq 2p+1$, except possibly for $k = 5$. 

  \item When $b = p$ and $p^2 \> | \> r-b$, the theorem shows that $\bar{V}_{k,a_p}$ is always reducible 
        if $p \geq 5$ 
        (and under the hypothesis \eqref{hyp} when $p = 3$). This is a new phenomenon not 
        occurring for slopes in $(0,1)$. 
        When $b = 3$, $v(a_p) = \frac{3}{2}$ and \eqref{hyp} fails, we expect that $\bar{V}_{k,a_p}$ might also be reducible in
        some cases, by analogy with the main result of \cite{Buzzard-Gee13}.
  \item Fix $k,a_p$ and  $b=b(k)$  as in Theorem~\ref{theorem-intro}. 
        Then the theorem implies the following local constancy result: 
        for any other weight $k^\prime\geq 2p+2$ with $k^\prime \equiv k\mod p^{1+v(b)}(p-1)$, 
        the reduction $\bar V_{k^\prime,a_p}$ is isomorphic to $\bar V_{k,a_p}$, 
        except possibly if $v(a_p)=\frac{3}{2}$ and $b = 3$. 
        We refer to \cite[Thm. B]{Berger12} for a general local 
        constancy result for any positive slope.
\end{itemize}
\end{remark}

The proof of Theorem~\ref{theorem-intro} uses the $p$-adic and mod $p$ Local Langlands Correspondences 
due to Breuil, Berger, Colmez, Dospinescu, Pa\v{s}k\=unas \cite{Breuil03a}, \cite{[Br03]}, \cite{Berger-Breuil10}, \cite{Colmez10}, 
\cite{Colmez-Dospinescu-Paskunas14} and an important compatibility between them with respect 
to the process of reduction \cite{Berger10}. The general strategy is due to Breuil and Buzzard-Gee and is outlined in 
\cite{[Br03]}, \cite{[BG09]}, \cite{[GG15]}. We briefly recall it now and explain several 
new obstacles we must surmount along the way.

Let $G = \mathrm{GL}_2(\Q_p)$, $K = \mathrm{GL}_2(\Z_p)$ be the standard maximal compact subgroup of
$G$ and $Z = \Q_p^*$ be the center of $G$.  Consider the locally algebraic representation of $G$ 
given by compact induction
\begin{eqnarray*}
   \Pi_{k, a_p} = \frac{ \mathrm{ind}_{{KZ}}^{G} \Sym^r \bar\Q_p^2 }{T-a_p},
\end{eqnarray*}
where $T$ is the Hecke operator, and consider the lattice in $\Pi_{k,a_p}$ given by 
\begin{equation}
\label{definetheta}  
  \Theta_{k, a_p} := \mathrm{image} \left( \mathrm{ind}_{KZ}^{G} \Sym^r \bar\Z_p^2 
                                           \rightarrow \Pi_{k, a_p} \right) 
  \iso \frac{ \mathrm{ind}_{KZ}^{G} \Sym^r \bar\Z_p^2 }{(T-a_p)(\mathrm{ind}_{KZ}^{G} \Sym^r \bar\Q_p^2) \cap \mathrm{ind}_{KZ}^{G} \Sym^r \bar\Z_p^2 }.
\end{equation}
It is known that the semisimplification of the reduction of this lattice satisfies 
$\bar\Theta_{k,a_p}^\mathrm{ss} \iso LL(\bar{V}_{k,a_p})$, where $LL$ is the (semi-simple) 
mod $p$ Local Langlands Correspondence of Breuil \cite{[Br03]}. 
One might require here the conditions $a_p^2 \neq 4p^{k-1}$ and $a_p \neq \pm (1+p)p^{(k-2)/2}$, but these 
clearly hold  if $k \geq 2p+2$ and $v(a_p) < 2$, see \cite{Berger-Breuil10}. 
By the injectivity of the mod $p$ Local Langlands correspondence,  $\bar\Theta_{k,a_p}^\mathrm{ss}$ determines 
$\bar{V}_{k,a_p}$ completely, and so it suffices to compute $\bar{\Theta}_{k,a_p}$. 

Let  $V_r = \Sym^r \bar{\F}_p^2$
be the usual symmetric power representation of $\Gamma:=\GL_2(\F_p)$ (hence of $KZ$, with $p \in Z$ acting trivially). 
Clearly there is a surjective map
\begin{eqnarray}
    \label{natural-map}
    \mathrm{ind}_{{KZ}}^{G} V_r \twoheadrightarrow \bar\Theta_{k,a_p},
\end{eqnarray}
for $r = k-2$. 
Write $X_{k,a_p}$ for the kernel. 
A model for $V_r$ is the space of  all homogeneous polynomials of degree $r$ in the two variables $X$ and $Y$ over 
$\bar{\F}_p$ with the standard action of $\Gamma$. 
Let  $X_{r-1} \subset V_r$ be the $\Gamma$- (hence $KZ$-) submodule generated by $X^{r-1}Y$. 
Let $V_r^*$ and $V_r^{**}$ be the submodules of $V_r$ consisting of polynomials divisible by $\theta$ and $\theta^2$ respectively,
for $\theta  := X^p Y - X Y^p$. 
 If $r \geq 2p+1$, then Buzzard-Gee have shown \cite[Rem. 4.4]{[BG09]}:
\begin{itemize}  
  \item   $v(a_p) > 1 \implies \mathrm{ind}_{KZ}^{G} \> X_{r-1} \subset X_{k,a_p}$,
  \item   $v(a_p) < 2 \implies \mathrm{ind}_{KZ}^{G} \> V_r^{**} \subset X_{k,a_p}$.
\end{itemize}
The  proof of Theorem~\ref{theorem-intro} for $k = 2p+2$ is known (cf. \cite[\S 2]{[GG15]}) and involves slightly different techniques, so for the
rest of this introduction assume that $r \geq 2p+1$. 
It follows that when $1 < v(a_p) <2$, the map (\ref{natural-map}) 
induces a surjective map $\mathrm{ind}_{KZ}^{G} \> Q \twoheadrightarrow \bar\Theta_{k,a_p}$,
where  
\begin{eqnarray*}
  Q := \frac{V_r}{X_{r-1} + V_r^{**}}.
\end{eqnarray*}  

To proceed further, one needs to understand the `final quotient' $Q$. It is not hard to see that 
{\it a priori} $Q$ has up to 3 Jordan-H\"older factors as a $\Gamma$-module.
The exact structure of $Q$ is derived in  \S \ref{a=1} to \S \ref{a=3..} by giving 
a complete description of the submodule $X_{r-1}$ and understanding to what extent it intersects with 
$V_r^{**}$. 
When $0 < v(a_p) < 1$, the relevant `final quotient' in \cite{[BG09]} is 
always irreducible allowing the authors to compute the reduction (up to separating out some reducible 
cases) using the useful general result \cite[Prop. 3.3]{[BG09]}.
When $1 < v(a_p) < 2$,  we show 
$Q$ is irreducible if and only if 
\begin{itemize}
  \item $b = 2$, $p \nmid r(r-1)$ or $b = p$, $p \nmid r-b$ 
\end{itemize}
and we obtain $\bar{V}_{k,a_p}$ immediately in these cases (Theorem \ref{thm1}). 

Generically, the quotient $Q$ has length 2 when $1 < v(a_p) < 2$. In fact,
we show that $Q$ has exactly two Jordan-H\"older factors, say $J$ and $J^\prime$, in the cases complementary to those
above
\begin{itemize}
  \item $b = 2$, $p \>|\>  r(r-1)$ or $b = p$, $p \> | \> r-b$,
\end{itemize}
as well as in the generic case
\begin{itemize}
  \item $3 \leq b \leq p-1$ and $p \nmid r-b$.
\end{itemize}
We now use the Hecke operator $T$  to `eliminate' one of $J$ or $J^\prime$. Something similar was done 
in \cite{[Br03]} and \cite{[GG15]} for bounded weights. 
That this can be done for all weights is one of the new contributions of this paper (see \S \ref{elimination}). 
It involves constructing certain rational functions 
$f\in\mathrm{ind}_{KZ}^G\Sym^r\bar\Q_p^2$, such that $(T-a_p)f
\in  \mathrm{ind}_{KZ}^G\Sym^r\bar\Z_p^2$ is integral, with reduction mapping to a simple function in 
say $\mathrm{ind}_{KZ}^G \,J^\prime$ that generates this last space of functions as a $G$-module.  
As $(T-a_p)f$  lies in the 
denominator of the expression \eqref{definetheta} describing $\bar\Theta_{k,a_p}$,
its reduction 
lies in $X_{k,a_p}$. Thus we obtain a surjection 
$\mathrm{ind}_{KZ}^G \, J \twoheadrightarrow\bar{\Theta}_{k,a_p}$  
and can apply \cite[Prop. 3.3]{[BG09]} again. 
For instance, let $$J_1=V_{p-b+1}\otimes D^{b-1} \text{ and } J_2=V_{p-b-1}\otimes D^b,$$ 
where $D$ denotes the determinant character. Then in 
the latter (generic) case above,  $Q \cong J_1 \oplus J_2$  is a direct sum. 
We construct a function $f$ which eliminates $J^\prime=J_2$ so that $J = J_1$ survives, 
showing that $\bar{V}_{k,a_p}\cong\mathrm{ind}(\omega_2^{b+p})$ (Theorem \ref{elimination1}).

The situation is more complicated when $Q$ has $3$ Jordan-H\"older factors, namely $J_0 = V_{b-2} \otimes D$, 
in addition to $J_1$  and $J_2$ above.
That this happens at all came as a surprise to us since it did not happen in the range 
of weights considered in \cite{[GG15]}. We show that this happens for 
the first time when $r = p^2-p+3$, 
and in general whenever 
\begin{itemize}
  \item $3 \leq b \leq p-1$ and $p \> | \> r-b$. 
\end{itemize}
This time we construct functions $f$ killing $J_0$ and 
$J_1$ (except when $b = 3$, $v(a_p) = \frac{3}{2}$ and $v(a_p^2-p^3) > 3$), so that $J_2$ survives instead,
and the reduction becomes $\mathrm{ind}(\omega_2^{b+1})$ (Theorems~\ref{elimination3}, \ref{elimination6}).  
Since $J_2$ was also the `final quotient' in \cite{[BG09]}, 
the reduction in these cases is the same as the generic answer obtained for slopes in $(0,1)$. 

As a final twist in the tale, we remark that even though one can eliminate all but 1 Jordan-H\"older factor $J$, 
one  needs to further separate out the reducible cases 
when $J = V_{p-2} \otimes D^n$, for some $n$. This happens in three cases:
\begin{itemize}
  \item 
        $b = 3$, $p \nmid r-b$,
  \item $b = p = 3$, $p \> || \> r-b$,
 \item  
        $b = p$, $p^2 \> | \> r-b$. 
\end{itemize}
In \S\ref{sep} we construct additional functions and use them to show that the map 
$\mathrm{ind}_{KZ}^G \, J \twoheadrightarrow \bar\Theta_{k,a_p}$  factors either through the cokernel of 
$T$ or the cokernel of $T^2-cT+1$, for some $c \in \bar\F_p$, and then apply the mod $p$ Local Langlands Correspondence directly to compute 
$\bar{V}_{k,a_p}$, as was done in \cite{[Br03]}, \cite{Buzzard-Gee13}.
In the first two cases, we show that the map above factors through the cokernel of $T$ so that the 
reducible case never occurs. We work under the assumption \eqref{hyp}, namely if $v(a_p)=\frac{3}{2}$, then
$v(a_p^2 - {r-1 \choose 2}(r-2)p^3)$ is equal to 3, which is the generic sub-case 
(Theorem~\ref{reducible1}).
On the other hand, in the third case we show that if $p \geq 5$ or if $p=3=b$ and \eqref{hyp}
holds, then the map factors through the cokernel of $T^2+1$, so that 
$\bar{V}_{k,a_p}$ is reducible and is as in  Theorem~\ref{theorem-intro} (Theorem~\ref{reducible2}).

One of the key ingredients that go into the proof of Theorem~\ref{theorem-intro} is  
a complete description of the 
structure of the submodule $X_{r-1}$ of $V_r$. 
We give its structure now as the result might be of some independent interest. 
To avoid technicalities, we state the following theorem in a weaker form than what we actually prove.
Let $ M:=M_2(\F_p)$ and consider 
$V_r$ as an $M$-module, with the obvious extension of the action of $\Gamma=\GL_2(\F_p)$ on it. 

\begin{theorem}
  \label{structure-of-Xrminusone}
  Let $p \geq 3$. Let $r \geq 2p+1$ and let $X_{r-1} =  \langle X^{r-1}Y \rangle$ be the $M$-submodule of 
  $V_r$ generated by the second highest monomial. Then $2 \leq \mathrm{length} \> X_{r-1} \leq 4$ 
  as an $M$-module. More precisely, if $2 \leq b \leq p-1$, then 
  $X_{r-1}$ fits into the exact sequence of $M$-modules
  \begin{eqnarray*}  
    V_{p-b+1} \otimes D^{b-1} \oplus V_{p-b-1} \otimes D^{b} \rightarrow X_{r-1} \rightarrow 
    V_{b-2} \otimes D \oplus V_{b} \rightarrow 0,
  \end{eqnarray*}
  and if $b = p$, then $X_{r-1}$ fits into the exact sequence of $M$-modules
  \begin{eqnarray*}
     V_{1}\otimes D^{p-1} \rightarrow X_{r-1} \rightarrow W \rightarrow 0,
  \end{eqnarray*}
  where $W$ is quotient of the length $3$ module $V_{2p-1}$.
\end{theorem}
Theorem~\ref{structure-of-Xrminusone} is 
proved for representations defined over $\F_p$ in \S\ref{a=1}  and \S\ref{structure} using results of Glover   
\cite{[G78]}. Here we have stated the corresponding result after extending scalars to $\bar\F_p$. We recall that
$V_r^*$ is the largest singular $M$-submodule of $V_r$ \cite[(4.1)]{[G78]}. It is the $M$-module structure 
of $X_{r-1}$ given in the theorem rather than just the $\Gamma$-module structure that plays a key role  
in understanding how $X_{r-1}$ intersects with $V_r^*$ and $V_r^{**}$. 

A more precise description of the structure of $X_{r-1}$ can be found in Propositions~\ref{JHfactors1} and \ref{JHfactors}. 
There we show that the  Jordan-H\"older factors in Theorem~\ref{structure-of-Xrminusone} that actually occur in $X_{r-1}$
are completely determined by the sum of the $p$-adic digits of an integer related to $r$.
As a corollary, we 
obtain the following curious formula for the dimension of $X_{r-1}$ in all cases.

\begin{cor}
Let $p \geq 3$ and let $r \geq 2p+1$. Write $r = p^n u$, with $p \nmid u$. Set $\delta = 0$ if $r = u$ and 
$\delta = 1$ otherwise. Let $\Sigma$ be the sum of the $p$-adic digits of $u -1$ in its base $p$ 
expansion. Then
  \begin{eqnarray*}
    \dim X_{r-1} = \begin{cases}
                    2 \Sigma + 2 + \delta(p+1-\Sigma), & \text{if } \Sigma \leq p - 1 \\
                    2p + 2,                            & \text{if } \Sigma > p - 1.  
                  \end{cases}
  \end{eqnarray*}
\end{cor}

\section{Basics}
\subsection{Hecke operator $T$ }
  \label{Hecke}
Recall $G = \GL_2(\Q_p)$ and $KZ = \GL_2(\Z_p) \Q_p^*$ is 
the standard compact mod center subgroup of $G$.
Let $R$ be a $\Z_p$-algebra and let $V = \Sym^r R^2\otimes D^s$ be the usual 
symmetric power representation of $KZ$ twisted by a power of the determinant character $D$ (with $p \in Z$ acting trivially), 
modeled on homogeneous polynomials
of degree $r$ in the variables $X$, $Y$ over $R$. For $g \in G$, $v \in V$, let $[g,v] \in \mathrm{ind}_{KZ}^{G} V$ be the
function with support in ${KZ}g^{-1}$ given
by 
$$g' \mapsto
\begin{cases}
  g'g \cdot v  \ & \text{ if } g' \in {KZ}g^{-1} \\
   0                  & \text{ otherwise.}
\end{cases}
$$ Any function in $\mathrm{ind}_{KZ}^G V$ is a finite linear combination of functions of the form $[g,v]$, for $g\in G$ and $v\in V$.
The Hecke operator  $T$ is defined by its action on these elementary functions via
\begin{equation}\label{T} T([g,v])=\underset{\lambda\in\F_p}\sum\left[g\left(\begin{smallmatrix} p & [\lambda]\\
                                                 0 & 1
                                                \end{smallmatrix}\right),\:v\left(X, -[\lambda]X+pY\right)\right]+\left[g\left(\begin{smallmatrix} 1 & 0\\
                                                                                                                                                   0 & p
                                                \end{smallmatrix}\right),\:v(pX,Y)\right],\end{equation}
where $[\lambda]$ denote the Teichm\"uller representative of $\lambda \in \F_p$.                                                
We will always denote the Hecke operator acting on $\mathrm{ind}_{KZ}^G V$ for various choices of 
$R=\bar\Z_p$, $\bar\Q_p$ or $\bar\F_p$ and for different values of $r$ and $s$ by $T$ as the underlying space will be clear from the context.                                               

\subsection{The mod $p$ Local Langlands Correspondence}

Let $V$ be a weight, i.e., an irreducible representation of $\GL_2(\F_p)$, thought of as a 
representation of $KZ$ by
inflating to $\GL_2(\Z_p)$ and making $p \in \Q_p^*$ act trivially. Let $V_r = \Sym^r \bar\F_p^2$ be the $r$-th 
symmetric power of the standard two-dimensional representation of $\GL_2(\F_p)$ on $\bar\F_p^2$. The set of weights  $V$
is exactly the set of modules $V_r \otimes D^i$, for $0 \leq r \leq p-1$ and $0 \leq i \leq p-2$. 
For $0 \leq r \leq p-1$, $\lambda \in \bar{\F}_p$ and $\eta : \Q_p^* \rightarrow \bar\F_p^*$ a smooth character, let
\begin{eqnarray*}
  \pi(r, \lambda, \eta) & := & \frac{\mathrm{ind}_{KZ}^{G} V_r}{T-\lambda} \otimes (\eta\circ \mathrm{det})
\end{eqnarray*}
be the smooth admissible representation of ${G}$, where $\mathrm{ind}_{KZ}^G$ is compact
induction, and  $T$ is the Hecke operator defined above; it generates the Hecke algebra 
$\mathrm{End}_{G} (\mathrm{ind}_{{KZ}}^{G} \> V_r) = \bar{\F}_p[T]$.
With this notation, Breuil's semisimple mod $p$ Local Langlands Correspondence 
\cite[Def. 1.1]{[Br03]} 
is given by:
\begin{itemize} 
  \item  $\lambda = 0$: \quad
             $\mathrm{ind}(\omega_2^{r+1}) \otimes \eta \:\:\overset{LL}\longmapsto\:\: \pi(r,0,\eta)$,
  \item $\lambda \neq 0$: \quad
             $\left( \omega^{r+1} \mathrm{unr}(\lambda) \oplus \mathrm{unr}(\lambda^{-1}) \right) \otimes \eta 
                   \:\:\overset{LL}\longmapsto\:\:  \pi(r, \lambda, \eta)^\mathrm{ss} \oplus  \pi([p-3-r], \lambda^{-1}, \eta \omega^{r+1})^\mathrm{ss}$, 
\end{itemize}
where $\{0,1, \ldots, p-2 \} \ni [p-3-r] \equiv p-3-r \mod p-1$.   
It is clear from the classification of  smooth admissible 
irreducible representations of ${G}$ by Barthel-Livn\'e \cite{Barthel-Livne94} and Breuil \cite{Breuil03a},  that this correspondence  is not 
surjective. However, the map ``${LL}$'' above is an injection and so it is enough to know ${LL}(\bar V_{k,a_p})$ to determine $\bar{V}_{k,a_p}$.

\subsection{Modular representations of $M$ and $\Gamma$}\label{basics}
 In order to make use of results in Glover \cite{[G78]}, let us abuse notation a bit and let $V_r$ be the space of
homogeneous polynomials $F(X,Y)$  in two variables $X$ and $Y$ of degree $r$ with coefficients in the finite field $\F_p$,
rather than in $\bar\F_p$. For the next few sections (up to \S\ref{a=3..}) we similarly
consider all subquotients of $V_r$ as representations defined over $\F_p$. This is not so serious, as once we have established
the structure of $X_{r-1}$ or $Q$ over $\F_p$, it immediately implies the corresponding result over $\bar\F_p$, by extension 
of scalars. The matrix 
algebra $M=\mathrm{M}_2(\mathbb{F}_p)$ acts on $V_r$ by the formula $$\left( \begin{smallmatrix} a & b\\
                                                             c & d
                                                            \end{smallmatrix} \right) \cdot F(X,Y) = F(aX+cY, bX+dY).$$
One has to be careful with the notation  while using the results of \cite{[G78]} as Glover indexed the symmetric power representations by dimension, instead of the degree of the
polynomials involved. In this paper, $V_r$ always has dimension $r+1$.
                                                           
We denote the set of singular matrices by $N \subseteq M$.  
An $\mathbb{F}_p[M]$-module $V$ is called \textquoteleft singular\textquoteright, if  each matrix $t\in N$ annihilates $V$,
i.e., if $t\cdot V=0$, for all $ t\in N$. The largest singular submodule of an arbitrary $\F_p[M]$-module $V$ is denoted by $V^*$. Note that any 
$M$-linear map must take a singular submodule (of its domain) to a singular submodule (of the range).  This simple observation will be very useful for us.

Let $X_r$ and $X_{r-1}$ be the $\mathbb{F}_p[M]$-submodules of $V_r$ generated by the  monomials $X^r$ and $X^{r-1}Y$ respectively. 
One can check that $X_r \subset X_{r-1}$ and are spanned  by the sets
$\{X^r, (kX+Y)^r : k\in\mathbb{F}_p\}$ and  $\{X^r, Y^r, X(kX+Y)^{r-1}, (X+lY)^{r-1}Y : k, l \in \mathbb{F}_p\}$ respectively 
\cite[Lem. 3]{[GG15]}. Thus we have $\dim X_r\leq p+1$ and
$\dim X_{r-1}\leq 2p+2$.  We will describe the explicit structure of the modules $X_r$ and $X_{r-1}$, according to 
the different congruence classes $a\in\{1,2,\ldots,p-1\}$ with $r\equiv a\mod (p-1)$. It will also be convenient to 
use the representatives $b \in \{2,\ldots, p-1,p\}$ of the congruence classes of $r \mod (p-1)$.

For $s\in\mathbb{N}$, we denote the sum of the $p$-adic digits of $s$ in its base $p$ expansion by $\Sigma_p(s)$. It is easy to see that 
$\Sigma_p(s)\equiv s\mod (p-1)$,  for any $s\in\mathbb{N}$. Let us write $r=p^nu$, where $n=v(r)$ and hence 
$p\nmid u$. In the study of the module $X_{r-1}$, the sum $\Sigma_p(u-1)$ plays a key role. For $r\equiv a \mod (p-1)$, 
observe that the sum $\Sigma_p(u-1)\equiv a-1\mod (p-1)$, therefore it 
varies discreetly over the infinite set $\{a-1,\, p+a-2,\,2p+a-3,\cdots\}$.

Let $\theta = \theta(X,Y)$ denote the special polynomial
$X^pY-XY^p$.  For $r\geq p+1$, 
we know  that \cite[(4.1)]{[G78]}
$$V_r^*:=\{F \in V_r : \theta \mid F\} \cong\begin{cases}
                                             0, &\text{if } r\leq p\\ 
                                             V_{r-p-1}\otimes D, &\text{if } r\geq p+1
                                             \end{cases}$$
                                            is the largest singular submodule of $V_r$. 
We define  $V_r^{**}$, another important submodule of $V_r$, by 
$$V_r^{**} := \{F \in V_r : \theta^2 \mid F\} \cong
                                              \begin{cases} 
                                              0, &\text{ if } r<2p+2\\
                                              V_{r-2p-2}\otimes D^2,  &\text{ if } r\geq 2p+2.  
                                              \end{cases}$$
Note that $V_r^{**}$ is obviously {\it{not}} the largest 
singular submodule of $V_r^*$. 

Next we introduce the submodules $X_r^*:=X_r\cap V_r^*$, $X_r^{**}:=X_r\cap V_r^{**}$, $X_{r-1}^*:=X_{r-1}\cap V_r^*$ and $X_{r-1}^{**}:=X_{r-1}\cap V_r^{**}$.
It follows that $X_r^*$ and $X_{r-1}^*$ are the largest singular submodules inside $X_r$ and $X_{r-1}$ respectively.
The group $\mathrm{GL}_2(\mathbb{F}_p)\subseteq M$ is denoted by $\Gamma$. For $r\geq 2p+1$, we will study  the $\Gamma$-module structure of 
$$Q:=\dfrac{V_r}{X_{r-1}+V_r^{**}}\, .$$ We will be particularly interested in the bottom row of the following commutative diagram of $M$-linear (hence also $\Gamma$-linear) maps:

\begin{eqnarray}\label{bigdiadram}
    \xymatrix{
    & 0 \ar[d] & 0 \ar[d] & 0 \ar[d] \\
    0 \ar[r] & \dfrac{X_{r-1}^*}{X_{r-1}^{**}}\ar[d]\ar[r] & \dfrac{X_{r-1}}{X_{r-1}^{**}} \ar[d] \ar[r] & \dfrac{X_{r-1}}{X_{r-1}^*} \ar[d]\ar[r] & 0 \\
    0 \ar[r] & \dfrac{V_r^*}{V_r^{**}} \ar[r]\ar[d]  & \dfrac{V_r}{V_r^{**}} \ar[d]\ar[r] & \dfrac{V_r}{V_r^*} \ar[r]\ar[d] & 0 \\
    0 \ar[r] & \dfrac{V_r^*}{V_r^{**}+X_{r-1}^*} \ar[r]\ar[d] & Q \ar[r]\ar[d] & \dfrac{V_r}{V_r^*+X_{r-1}} \ar[d]\ar[r] & 0. \\
               &0 & 0 & 0}
\end{eqnarray}

\begin{prop}\label{es1}
Let $p\geq 3$ and $r\geq p$ with $r\equiv a\mod (p-1)$, for $1\leq a\leq p-1$. Then the $\Gamma$-module structure of $V_r/V_r^*$ is 
given by
\begin{eqnarray}
\label{ses*} 
    0\rightarrow V_{a}\rightarrow\dfrac{V_r}{V_r^*}\rightarrow V_{p-a-1}\otimes D^{a}\rightarrow 0.
\end{eqnarray} 
The sequence splits as a sequence of $\,\Gamma$-modules if and only if $a=p-1$.
\end{prop}

\begin{proof} For $r\geq p$, we obtain that $V_r/V_r^*\cong V_{a+p-1}/V_{a+p-1}^*$,  using \cite[(4.2)]{[G78]}.
The exact sequence then follows from  \cite[Lem. 5.3]{[Br03]}. Note that it must split when $a=p-1$, as $V_{p-1}$ is an injective $\Gamma$-module.
The fact that it is non-split for the other congruence classes can be derived from the $\Gamma$-module structure of $V_{a+p-1}$ (see, e.g., \cite[ (6.4)]{[G78]} or \cite[Thm. 5]{[GG15]}). 
 %
 %
 \end{proof}

\begin{prop}\label{es2}
 Let $p\geq 3$ and $2p+1\leq r\equiv a\mod (p-1)$, with $1\leq a\leq p-1$, the $\Gamma$-module structure of $V_r^*/V_r^{**}$ is given by
 \begin{eqnarray}
 \label{ses**} 0\rightarrow V_{p-2}\otimes D\rightarrow\dfrac{V_r^*}{V_r^{**}}\rightarrow V_{1}\rightarrow 0,  &\text{if } a=1,\\
 \label{ses**2}0\rightarrow V_{p-1}\otimes D\rightarrow\dfrac{V_r^*}{V_r^{**}}\rightarrow V_{0}\otimes D\rightarrow 0,  &\text{if } a=2,\\
 \label{ses**3}0\rightarrow V_{a-2}\otimes D\rightarrow\dfrac{V_r^*}{V_r^{**}}\rightarrow V_{p-a+1}\otimes D^{a-1}\rightarrow 0,  &\text{if } 3\leq a\leq p-1,
\end{eqnarray} and the sequences split if and only if $a=2$.
\end{prop}
\begin{proof}
 We use \cite[(4.1)]{[G78]} to get that $V_r^*/V_r^{**}\cong (V_{r-p-1}/V_{r-p-1}^*)\otimes D$. 
Since $p\leq r-p-1$ by hypothesis, we apply Proposition \ref{es1} to deduce the
 $\Gamma$-module structure of $(V_{r-p-1}/V_{r-p-1}^*)\otimes D$.
\end{proof}

The following lemma will be used many times throughout the article, to determine if certain polynomials $F \in V_r$ 
are divisible by 
$\theta$ or $\theta^2$. We skip the proof since it is elementary. 
\begin{lemma}\label{elelemma}
 Suppose $F(X,Y)=\underset{\substack{0\leq j\leq r}}\sum c_j\cdot X^{r-j}Y^j\in \mathbb{F}_p[X,Y]$ is such that $c_j\neq 0$ implies that 
 $j\equiv a \mod (p-1)$, for some fixed $a \in\{1,2,\cdots,p-1\}$. Then
 \begin{enumerate}
  \item[(i)] $F\in V_r^*$ if and only if $c_0 = c_r =0$ and $\underset{j}\sum c_j=0$ in $\mathbb{F}_p$.
  \item[(ii)] $F\in V_r^{**}$ if and only if $c_0 = c_1 = c_{r-1} = c_r= 0$ and $\underset{j}\sum c_j = \underset{j}\sum jc_j= 0$ in $\mathbb{F}_p$.
 \end{enumerate}
\end{lemma}
\subsection{Reduction of binomial coefficients}
In this article, the mod $p$ reduction of binomial coefficients plays a very important role. We will repeatedly use the following theorem and often 
refer to it as Lucas' theorem, as it was proved by E. Lucas in 1878.  
\begin{theorem}\label{lucas}
 For any prime $p$, let $m$ and $n$ be two non-negative integers with  $p$-adic expansions $m=m_kp^k+m_{k-1}p^{k-1}+\cdots +m_0$ and $n=n_kp^k+n_{k-1}p^{k-1}+\cdots +n_0$ respectively.
 Then 
 $\binom{m}{n}\equiv \binom{m_k}{n_k}\cdot\binom{m_{k-1}}{n_{k-1}}\cdots\binom{m_0}{n_0}\mod p,$ with the convention that $\binom{a}{b}=0$, if $b>a$.
\end{theorem}
The following elementary congruence mod $p$ will also be used in the text. For any $i \geq 0$,
\begin{eqnarray*}
\underset{k=0}{\overset{p-1}{\sum}} k^i \equiv
                                              \begin{cases}
                                               -1, & \text{ if } i=n(p-1), \text{ for some } n \geq 1, \\
                                               0,  & \text{ otherwise (including the case $i=0$, as $0^0=1$)}.
                                              \end{cases}
\end{eqnarray*}
This follows from the following frequently used fact in characteristic zero. For any $i \geq 0$,
\begin{equation}
\label{identity}
\underset{\lambda\in\F_p}\sum[\lambda]^i = 
                                             \begin{cases}
                                               p, &\text{if }i=0,\\
                                               p-1, &\text{if}\: i=n(p-1)\: \text{for some}\: n\geq 1,\\
                                               0, & \text{if } (p-1)\nmid i,
                                             \end{cases}
\end{equation} 
where $[\lambda] \in \Z_p$ is the Teichm\"uller representative of $\lambda \in \F_p$.

We now state some important congruences, leaving the proofs to the reader as exercises. These technical lemmas are used in 
checking the criteria given in Lemma \ref{elelemma}, and also in constructing functions $f\in\mathrm{ind}_{KZ}^G\,\mathrm{Sym}^r\bar{\Q}_p^2$ with 
certain desired properties (cf. \S \ref{combinatorial}, \S\ref{elimination} and \S \ref{sep}).
\begin{lemma}
\label{comb1}
 For $r\equiv a\mod (p-1)$ with $1\leq a\leq p-1$, we have $$S_r:=\sum_{\substack{0\,<j\,<\,r,\\ j\,\equiv\, a\mod(p-1)}}\binom{r}{j}\equiv 0\mod p.$$
 Moreover, we have $\dfrac{1}{p}S_r \equiv \dfrac{a-r}{a}\mod p$,  for $p>2$.
\end{lemma}
\begin{lemma}\label{comb3a}
 Let  $r\equiv b\mod (p-1)$, with $2\leq b\leq p$.  Then we have
  $$T_r:=\sum_{\substack{0<\,j\,<\,r-1,\\ j\equiv\, b-1\mod p-1}} \binom{r}{j}\equiv (b-r)\mod p.$$
  \end{lemma}
\begin{lemma}\label{comb4}
 Let $ p\geq 3$, $r\equiv 1\mod (p-1)$, i.e., $b=p$ with the notation as above. 
 If $p\mid r$, then
 $$S_r:=\sum_{\substack{1<\,j\,<\,r,\\ j\equiv\, 1\mod (p-1)}} \binom{r}{j} = \sum_{\substack{0<\,j\,<\,r-1,\\ j\equiv\,0\mod (p-1)}} \binom{r}{j} \equiv (p-r)\mod p^2.$$
 \end{lemma}

\section{$r\equiv 1\mod (p-1)$}
\label{a=1}
 
In this section, we compute the Jordan-H\"older (JH) factors of $Q$ as a $\Gamma$-module, when $r\equiv 1\mod (p-1)$. This is the case 
$a=1$ and $b=p$, with the notation above. 

\begin{lemma}
\label{star elt}
 Let $p\geq 3$,  $r>1$ and let $r\equiv 1\mod (p-1)$. 
 \begin{enumerate}
 \item[(i)] If $p\nmid r$, then $X_r^*/ X_r^{**}\cong V_{p-2}\otimes D$, as a $\Gamma$-module.
 \item[(ii)] If $p\mid r$, then $X_r^*/ X_r^{**}= 0$.
\end{enumerate}
\end{lemma}

\begin{proof}
 We have
 \begin{enumerate}
 \item[(i)] Consider the polynomial $F(X,Y)=\underset{k\in\mathbb{F}_p}\sum (kX+Y)^r\in X_r$. 
 We have 
 $$F(X,Y)=\underset{j=0}{\overset{r}\sum} \binom{r}{j}\cdot\underset{k\in\mathbb{F}_p}\sum k^{r-j}\cdot X^{r-j}Y^j\equiv
 \underset{\substack{0\leq j\,<\,r,\\ j\equiv\, 1\mod p-1}}\sum - \binom{r}{j}\cdot X^{r-j}Y^j \mod p.$$ 
 The sum of the coefficients of $F(X,Y)$ is congruent to
 $0\mod p$, by Lemma \ref{comb1}. Applying Lemma \ref{elelemma}, we get that $F(X,Y)\in V_r^*$. 
 As $p\nmid r$, the coefficient of $X^{r-1}Y$ in $F(X,Y)$ is $-r \not\equiv 0\mod p$. Hence 
 $F(X,Y)\notin V_r^{**}$, and so $F(X,Y)$ has non-zero image in $X_r^*/ X_r^{**}$. For $r=2p-1$, we have $0\neq X_r^*/X_r^{**}\subseteq V_r^*/V_r^{**}\cong V_{p-2}\otimes D$, which is irreducible
 and the result follows. 
 If $r\geq 3p-2$, then $V_r^*/V_r^{**}$ has dimension $p+1$, but 
  \cite[(4.5)]{[G78]} implies that $\dim X_r^*<\dim X_r\leq p+1$. So we have $0\neq X_r^*/X_r^{**}\subsetneq V_r^*/V_r^{**}$. 
 Now it follows from Proposition \ref{es2}  that $X_r^*/X_r^{**}\cong V_{p-2}\otimes D$. $\vspace{.05in}$

 \item[(ii)] Write $r=p^nu$, where $n\geq 1$ and $p\nmid u$. The map $\iota:X_u\rightarrow X_r$, 
 defined by $\iota (H(X,Y)) := H(X^{p^n},Y^{p^n})$, is a well-defined $M$-linear surjection from $X_u$ to $X_r$. It is also an injection, 
 as $H(X^{p^n},Y^{p^n}) = H(X,Y)^{p^n} \in \mathbb{F}_p[X,Y]$. Hence the $M$-isomorphism $\iota : X_u\rightarrow X_r$ must take $X_u^*$, 
 the largest singular submodule of $X_u$, isomorphically to $X_r^*$. 
 
 If $u=1$, then $X_r^*\cong X_u^*=0$, so $X_r^*=X_r^{**}$ follows trivially.
 If $u>1$, then as  $ p\nmid u\equiv r\equiv 1\mod (p-1)$, we get $u\geq 2p-1$ and $V_u^*\cong V_{u-p-1}\otimes D$. 
 For any $F \in X_r^*$, we have $F = \iota(H)$, for some $H \in X_u^*$. Writing $H = \theta H^\prime$ with $H^\prime\in V_{u-p-1}$,
 we get $F = \iota(H) = (\theta H^\prime)^{p^n}$. As $n\geq 1$, clearly $\theta^2$ divides  $F$. 
 Therefore $X_r^*\subseteq V_r^{**}$, equivalently $X_r^*=X_r^{**}$.
 \end{enumerate}
\end{proof}

The $p$-adic expansion of $r-1$ will play an important role in our study of the module $X_{r-1}$. Write 
\begin{equation}
 \label{C}r-1=r_mp^m+r_{m-1}p^{m-1}+\cdots r_ip^i,
\end{equation} where  $r_j\in\{0,1,\cdots ,p-1\}$, $m \geq i$ and $r_m$, $r_i \neq 0$. 
If $i>0$, then we let $r_j =0$, for $0\leq j\leq i-1$.  

With the notation introduced in Section \ref{basics},  we have $a=1$, so $\Sigma_p(r-1)\equiv 0\mod (p-1)$. Excluding the case $r=1$, note that the smallest possible value of $\Sigma_p(r-1)$ is $p-1$.
Also recall  that  the dimension of $X_{r-1}$ is bounded above by $2p+2$ and  a standard generating set is given by 
$\{X^r, Y^r, X(kX+Y)^{r-1},(X+lY)^{r-1}Y : k,l \in \mathbb{F}_p\}$, over $\mathbb{F}_p$.  
\begin{lemma}\label{relation}
 For $p\geq 2$, if $p\leq r\equiv 1\mod(p-1)$ and $\,\Sigma=\Sigma_p(r-1)=p-1$, then 
 \begin{eqnarray*}
  \underset{k=0}{\overset{p-1}{\sum}} X(kX+Y)^{r-1} \equiv -X^r \hspace{.2in} \text{and} \hspace{.2in}
  \underset{l=0}{\overset{p-1}{\sum}} (X+lY)^{r-1}Y \equiv -Y^r \mod p.
 \end{eqnarray*}
 As a consequence, $\dim X_{r-1}\leq 2p$.
 \end{lemma}

 \begin{proof}
  It is enough to show one of the congruences, since the other will then follow by applying the matrix 
$w=\left(\begin{smallmatrix}0&1\\ 1&0\end{smallmatrix}\right)$ to it.
  We compute that $$F(X,Y)=\underset{k=0}{\overset{p-1}{\sum}} X(kX+Y)^{r-1}\equiv\underset{\substack{0<s<r\\s\equiv 0\mod p-1}}{\sum} - \binom{r-1}{s}\cdot X^{s+1}Y^{r-1-s}\mod p.$$
 We claim that if $0< s < r-1$ and $s\equiv 0\mod (p-1)$, then $\binom{r-1}{s}\equiv 0\mod p$. The claim implies that
$F(X,Y)\equiv -\binom{r-1}{r-1}\cdot X^r\equiv -X^r\mod p$, as required.
 
 \underline{Proof of claim}: Let $ s = s_m p^m + \cdots +s_1p + s_0\,$
 be the $p$-adic expansion of $s <r-1$,
 where $m$ is as in the expansion \eqref{C} above. 
 Since $ s \equiv 0 \mod (p-1)$, we have $\,\Sigma_p(s)\equiv 0\mod (p-1)$ too.
 If  $\binom{r-1}{s} \not \equiv 0 \mod p$, then by  Lucas' theorem  $0 \leq s_j\leq r_j$, for all $j$. Taking the sum, 
 we get that $0 \leq \Sigma_p(s) \leq \Sigma = p-1$. But since $s>0$, $\Sigma_p(s)$ has to
 be a strictly positive multiple of $p-1$, and so it is $p-1$. Hence $s_j = r_j$, for all $j \leq m$, and we have $s = r-1$, 
 which is a contradiction.
 \end{proof}

We observe that $p\mid r$ if and only if $r_0=p-1$ in \eqref{C}.  Therefore if $\Sigma=r_0+\cdots r_m=p-1$, 
then the condition $p\mid r$ is equivalent to $r=p$. Our next proposition treats the case $\Sigma=p-1$, and to avoid
the possibility of $p$ dividing $r$, we exclude the case $r=p$. The fact that $p\nmid r$ will be used crucially in the proof. 
This does not matter, as eventually we wish to compute $Q$ for $r\geq 2p+1$.
 
\begin{prop}\label{prop1}
 For $p\geq 2$, if $p< r\equiv 1\mod (p-1)$ and $\,\Sigma=\Sigma_p(r-1)=p-1$, then 
 \begin{enumerate}\item[(i)] $X_{r-1}\cong V_{2p-1}$ as an $M$-module, and the $M$-module structure of $X_r$ is given by
 $$0\rightarrow V_{p-2}\otimes D\rightarrow X_r\rightarrow V_1\rightarrow 0.$$
  \item[(ii)] $X_{r-1}^*=X_r^*\cong V_{p-2}\otimes D$ \text{and } $X_{r-1}^{**}=X_r^{**}=0$. 
  \item[(iii)] For $r>2p$, $Q$ has only one JH factor $V_1$, as a $\,\Gamma$-module.
 \end{enumerate}
\end{prop}
\begin{proof}
 It is easy to check that $\{S(kS+T)^{2p-2}, (S+lT)^{2p-2}T : k, l \in\mathbb{F}_p\}$ gives a basis of $V_{2p-1}$ over $\mathbb{F}_p$. 
 We define an $\mathbb{F}_p$-linear map
 $\eta:V_{2p-1}\rightarrow X_{r-1}$, by $\eta\left(S(kS+T)^{2p-2}\right) = X(kX+Y)^{r-1}$ and 
 $\eta\left((S+lT)^{2p-2}T\right) = (X+lY)^{r-1}Y$, for $k,l\in\mathbb{F}_p$. We
 claim that the map $\eta$ is in fact an $M$-linear injection. By  Lemma \ref{relation} we then have
 \begin{equation}\label{X^r}\eta(S^{2p-1})=X^r,\: \:\eta(T^{2p-1})=Y^r.\end{equation}
 The $M$-linearity can be checked on the basis elements of $V_{2p-1}$ above by an elementary computation which uses the fact that 
 $r-1\equiv 0 \mod (p-1)$ and \eqref{X^r}, so we leave it to the reader.
 
 As a $\Gamma$-module, $\text{soc}(V_{2p-1})=V_{2p-1}^*\cong V_{p-2}\otimes D$ is irreducible.  
 Therefore if $\ker \eta \neq 0$, then it must
 contain the submodule $V_{2p-1}^*$. Consider 
 $H(S,T) = \underset{k=0}{\overset{p-1}{\sum}} \left( \begin{smallmatrix}
                                                                    k & 1\\
                                                                    1 & 0
                                               \end{smallmatrix} \right) 
  \cdot S^{2p-1}=(S^pT-ST^p)S^{p-2}\in V_{2p-1}^*$.
  By $M$-linearity, we have  $\eta(H) = F(X,Y) \in X_r^*\setminus X_r^{**}$, where $F$ is as in the proof of Lemma \ref{star elt} (i). 
  In particular, this shows that $H\notin \ker \eta$. As $V_{2p-1}^*\nsubseteq \ker \eta$, we have $\ker \eta =0$.   
   
  Thus $\eta:V_{2p-1}\rightarrow X_{r-1}$ is an injective $M$-linear map. 
  By Lemma \ref{relation}, $\dim X_{r-1} \leq 2p=\dim V_{2p-1}$, forcing $\eta$ to be an isomorphism.
  Therefore  the largest singular submodule $X_{r-1}^*$ inside $X_{r-1}$ has to be isomorphic to $V_{2p-1}^*
  \cong V_{p-2}\otimes D$, the largest singular submodule of $V_{2p-1}$. 
  Then Lemma \ref{star elt} (i) implies that $X_r^*$ is a non-zero submodule of $X_{r-1}^*\cong V_{p-2}\otimes D$, which is irreducible. 
  So we must have $X_r^*=X_{r-1}^*$.
  Again by Lemma \ref{star elt} (i),
  $X_{r-1}^{**} (\supseteq X_r^{**})$ is a proper submodule of $X_{r-1}^*$. 
  Hence $X_{r-1}^{**} =X_r^{**}=0$.
   
  Since $\dim (X_{r-1}/X_{r-1}^*) =p+1= \dim (V_r/V_r^*)$, the rightmost module in the bottom row 
  of Diagram \eqref{bigdiadram} is $0$.
  As the dimension of $X_{r-1}^{*}/X_{r-1}^{**}$ is $p-1$, the leftmost module must have dimension $2$. 
  It has to be $V_1$, as the short exact sequence \eqref{ses**} does not split for $p\geq 3$. 
  For $p=2$ and $r\geq 5$, the only two-dimensional quotient of $V_r^*/V_r^{**}$ is $V_1$, 
  as one checks that $V_r^*/V_r^{**}\cong V_1\oplus V_0$.
  Hence we get $Q \cong V_{1}$ as a $\Gamma$-module.  
\end{proof}

The next lemma about the dimension of $X_{r-1}$ is a special case of Lemma~ \ref{dim7}, 
proved at the beginning of Section \ref{structure}.    
\begin{lemma}
\label{dim1}
 For $p\geq 2$, suppose $ p\nmid r\equiv 1\mod (p-1)$. If $\,\Sigma=\Sigma_p(r-1) > p-1$, 
 then $\dim X_{r-1} =2p+2$.
\end{lemma}
  

\begin{lemma}\label{dim3}
 For any $r$, if $\dim X_{r-1} = 2p+2$, then $\dim X_r =p+1$.
\end{lemma}

\begin{proof}
 Suppose $X_r$ has dimension smaller than $p+1$. Then the standard spanning set of $X_r$ 
 is linearly dependent,
 i.e., there exist constants  $A, c_k\in\mathbb{F}_p$, for $k \in \{0,1,\ldots, p-1 \}$, not all zero, such that 
 $AX^r+\underset{k=0}{\overset{p-1}\sum}c_k(kX+Y)^r=0$, which implies that
 \begin{eqnarray*}
   AX^r + c_0Y^r + \underset{k=1}{\overset{p-1}\sum}kc_kX(kX+Y)^{r-1} + \underset{k=1}{\overset{p-1}\sum} c_kk^{r-1}(X+k^{-1}Y)^{r-1}Y = 0.
 \end{eqnarray*}
 But this shows that the standard spanning set $\{X^r, Y^r, X(kX+Y)^{r-1},(X+lY)^{r-1}Y : k,l \in \mathbb{F}_p\}$ of $X_{r-1}$ is linearly dependent, 
 contradicting the hypothesis $\dim X_{r-1}=2p+2$.
\end{proof}
 
For any $r$, let us set $r^\prime:=r-1$. 
The  trick introduced in \cite{[GG15]} of using the structure of $X_{r'}\subseteq V_{r'}$ to study  $X_{r-1}\subseteq V_r$ via the map $\phi$ described below,
turns out to be very useful in general.

\begin{lemma}\label{onto}
  There exists an $M$-linear surjection $\phi:X_{r^\prime}\otimes V_1\twoheadrightarrow X_{r-1}$.
\end{lemma}

\begin{proof}
 The map $\phi_{r^\prime,1}: V_{r^\prime}\otimes V_1\twoheadrightarrow V_r$ sending $u\otimes v\mapsto uv$, 
 for $u\in V_{r^\prime}$ and $v\in V_1$, is $M$-linear by \cite[(5.1)]{[G78]}.
 Let $\phi$ be its restriction to the $M$-submodule 
 $X_{r^\prime}\otimes V_1\subseteq V_{r^\prime}\otimes V_1$. 
 The module $X_{r^\prime}\otimes V_1$ is generated by 
 $X^{r^\prime}\otimes X$ and $X^{r^\prime}\otimes Y$, which map to $X^r$ and $X^{r-1}Y\in X_{r-1}$ respectively.
 So the image of $\phi$ lands in $X_{r-1}\subseteq V_r$. 
 The surjectivity  follows 
 as $X^{r-1}Y$ generates $X_{r-1}$.
\end{proof}
 
\begin{lemma}\label{Xrprime*}
 For $p\geq 3$, if $r\equiv 1\mod (p-1)$ with $\Sigma_p(r^\prime)>p-1$, then
 \begin{enumerate}
 \item[(i)] $X_{r^\prime}^{**}=X_{r^\prime}^*$ has dimension $1$ over $\mathbb{F}_p$. In fact, it is $M$-isomorphic to $D^{p-1}$.
 \item[(ii)] $\phi(X_{r^\prime}^*\otimes V_1)\subseteq V_r^{**}$ and $\phi(X_{r^\prime}^*\otimes V_1)\cong V_1\otimes D^{p-1}$. 
 \end{enumerate}
\end{lemma}

\begin{proof}
 Consider $F(X,Y):=X^{r^\prime}+\underset{k\in\mathbb{F}_p}\sum(kX+Y)^{r^\prime}\in X_{r^\prime}\subseteq V_{r^\prime}$. 
 It is easy to see that $$F(X,Y)\equiv -\underset{\substack{0<j<r^\prime\\j\equiv 0\mod (p-1)}}\sum\binom{r^\prime}{j}X^{r^\prime-j}Y^j\mod p.$$
 Using Lemmas \ref{elelemma} and \ref{comb1} 
 we check that $F(X,Y)\in V_{r'}^{**}$, for $p\geq 3$. Since $\Sigma_p(r^\prime)>p-1$ 
 or equivalently $\Sigma_p(r^\prime)\geq 2p-2$,  using Lucas' theorem one can show that at least one of the 
 coefficients $\binom{r^\prime}{j}$ above is non-zero  mod $p$. So we have $0\neq F(X,Y)\in X_{r^\prime}^{**}\subseteq X_{r^\prime}^*$.
   Since $r^\prime \equiv p-1\mod (p-1)$, \cite[(4.5)]{[G78]} gives the following
 short exact sequence of $M$-modules:
 \begin{equation}\label{ses1}
  0\rightarrow X_{r^\prime}^*\rightarrow X_{r^\prime}\rightarrow V_{p-1}\rightarrow 0.
 \end{equation}
 As $\dim X_{r^\prime} \leq p+1$ and $X_{r^\prime}^{**}\neq 0$, we must have $\dim X_{r'}^{**}=\dim X_{r^\prime}^* =1$.
 Hence $X_{r'}^{**}=X_{r^\prime}^*\cong D^n$ for some $n\geq 1$. Checking the action of diagonal matrices on
 $F(X,Y)$, we get $n=p-1$.
 
 As $X_{r'}^{**}=X_{r^\prime}^*$, each element of $X_{r'}^*$ is divisible by $\theta^2$.
 Therefore it follows from the definition of the map $\phi$ that $\phi(X_{r^\prime}^*\otimes V_1)\subseteq V_r^{**}$.
 For any non-zero $F\in X_{r^\prime}^*$, note
 that $\phi(F\otimes X)=FX\neq 0$. 
 We know that $X_{r^\prime}^*\otimes V_1\cong V_1\otimes D^{p-1}$ is irreducible of
 dimension $2$ and its image under $\phi$ is non-zero.  
 Hence $\phi(X_{r^\prime}^*\otimes V_1)\cong V_1\otimes D^{p-1}\subseteq X_{r-1}$. 
\end{proof}

\begin{prop}\label{prop2}
 Let $p\geq 3$, $r> 2p$ and $p\nmid r \equiv 1\mod(p-1)$. If $\,\Sigma =\Sigma_p(r-1)> p-1$, then
 \begin{enumerate}
 \item[(i)] The $M$-module structures of $X_{r-1}$ and $X_{r}$ are  given by the   exact sequences
   \begin{eqnarray*}
        & 0\rightarrow V_1\otimes D^{p-1}\rightarrow X_{r-1}  \rightarrow V_{2p-1}\rightarrow 0, \\
        & 0\rightarrow V_{p-2}\otimes D\rightarrow X_r  \rightarrow V_1\rightarrow 0.
   \end{eqnarray*}
 \item[(ii)] $X_r^*\cong V_{p-2}\otimes D$ and  $X_{r-1}^*\cong V_1\otimes D^{p-1}\oplus V_{p-2}\otimes D$. 
 \item[(iii)]$X_r^{**}=0$ and $X_{r-1}^{**}\cong V_1\otimes D^{p-1}$. 
 \item[(iv)] $Q\cong V_1$ as a $\Gamma$-module.
 \end{enumerate}
\end{prop}

\begin{proof}
 By Lemma \ref{dim1},  $\dim X_{r-1} =2p+2$, so by Lemma \ref{onto}, we must have $\dim X_{r^\prime}\otimes V_1 \geq 2p+2$. 
 This forces $X_{r^\prime}$ to have its highest possible dimension, namely, $p+1$. 
 Thus the $M$-map $\phi:X_{r^\prime}\otimes V_1\twoheadrightarrow X_{r-1}$ is actually an isomorphism.
 Tensoring the short exact sequence \eqref{ses1}  by $V_1$, we get the exact sequence
 \begin{equation*}
  0\rightarrow X_{r^\prime}^*\otimes V_1\rightarrow X_{r^\prime}\otimes V_1\rightarrow V_{p-1}\otimes V_1\rightarrow 0. 
 \end{equation*}
 The middle module is $M$-isomorphic to $X_{r-1}$, and the rightmost module is 
 $M$-isomorphic to $V_{2p-1}$, by \cite[(5.3)]{[G78]}. Thus the exact sequence reduces to
 \begin{equation}
 0\rightarrow X_{r^\prime}^*\otimes V_1\rightarrow X_{r-1}\rightarrow V_{2p-1}\rightarrow 0.
 \end{equation}
 Using Lemma \ref{Xrprime*} (i), we get the short exact sequence
 $$0\rightarrow V_1\otimes D^{p-1}\rightarrow X_{r-1}\rightarrow V_{2p-1}\rightarrow 0.$$

 Since $M$-linear maps must take singular submodules to singular submodules, the above sequence gives rise to the following exact sequence
 \begin{equation}
  0\rightarrow  V_1\otimes D^{p-1}\rightarrow X_{r-1}^*\rightarrow V_{2p-1}^*\cong V_{p-2}\otimes D.
 \end{equation}
 The rightmost module above is irreducible, so the map $X_{r-1}^*\rightarrow V_{p-2}\otimes D$ is either the zero map or it is a surjection.   

 By Lemma \ref{dim3}, $\dim X_r = p+1$ and so by \cite[(4.5)]{[G78]}, we have $\dim X_r^*=p-1$. 
 By Lemma \ref{star elt} (i), we get $X_r^{**}=0$ and $X_r^*\cong V_{p-2}\otimes D$ must be a JH factor of $X_{r-1}^*$. 
 Therefore the rightmost map above must be onto, as otherwise $X_{r-1}^*\cong X_{r^\prime}^*\otimes V_1\cong V_1\otimes D^{p-1}$. 
 So we have
 \begin{equation}
  0\rightarrow X_{r^\prime}^*\otimes V_1 \cong V_1\otimes D^{p-1}\rightarrow X_{r-1}^*\rightarrow  V_{p-2}\otimes D\rightarrow 0.
 \end{equation}
 Thus $X_{r-1}^*$ has two JH factors, of dimensions $2$ and $p-1$ respectively. Moreover, since  $X_r^*\cong V_{p-2}\otimes D$ is a submodule of 
 $X_{r-1}^*$, the sequence above must split, and we must have $$X_{r-1}^*=\phi(X_{r^\prime}^*\otimes
 V_1)\oplus X_r^*\cong V_1\otimes D^{p-1}\oplus V_{p-2}\otimes D.$$ 

 Knowing the structure of $X_{r-1}^*$ as above, next we want to see how the submodule $X_{r-1}^{**}$ sits inside it. 
 By Lemma \ref{Xrprime*}, we have $\phi(X_{r^\prime}^*\otimes V_1)\subseteq V_r^{**}$, on the other hand $X_r\cap V_r^{**}=X_r^{**}=0$. 
 Therefore $X_{r-1}^{**}=X_{r-1}^*\cap V_r^{**}=\phi(X_{r^\prime}^*\otimes V_1)\cong V_1\otimes D^{p-1}$  has dimension $2$.

 Now we count the dimension  $\dim Q = 2p+2 - \dim X_{r-1} +\dim X_{r-1}^{**} = 2$. 
 The final statement $Q\cong V_1$ follows from Diagram \eqref{bigdiadram} as in the proof of Proposition \ref{prop1}.
\end{proof}

Thus we know  $Q$ is isomorphic to  $V_1$, whenever $r$ is prime to $p$. Next we treat the case  $p$ divides $r$. 
Since $r\equiv 1\mod (p-1)$, we see that $r$ can be a pure $p$-power.
We will show that $Q$ has two JH factors as a $\Gamma$-module, irrespective of whether $r$ is a $p$-power or not. 
The following result about $\dim X_{r-1} $ when $p\mid  r$ is stated without proof,
as it follows from 
the more general Lemma \ref{gendim2}  in Section \ref{structure}. 
\begin{lemma}\label{dim2}
  Let $p\geq 2$ and $ r\equiv 1\mod (p-1)$. If $p\mid r$ but $r$ is not a pure $p$-power, then $\dim X_{r-1} =2p+2$.
\end{lemma}

\begin{lemma}\label{dim4}
 For $p\geq 2$ and $r=p^n$ with $n\geq 2$, $\dim X_{r-1}=p+3$.
\end{lemma}

\begin{proof}
 We know that $\Gamma=B\sqcup BwB$, where $B\subseteq \Gamma $ is the subgroup of upper-triangular matrices, and 
 $w = \left( \begin{smallmatrix}
                    0 & 1\\
                    1 & 0
             \end{smallmatrix} \right)$.
 Using this decomposition and the fact that $r=p^n$, one can see that the $\mathbb{F}_p[\Gamma]$-span of $X^{r-1}Y\in V_r$ 
 is generated by the set $\{X^r, X^{r-1}Y, Y^{r}, X(kX+Y)^{r-1} : k \in \mathbb{F}_p\}$ over 
 $\mathbb{F}_p$. We will show that this generating set is linearly independent. Suppose that
 $$AX^r+BY^r+DX^{r-1}Y+\underset{k=0}{\overset{p-1}\sum}c_k X(kX+Y)^{r-1}=0,$$
 where $A,B,D,c_k\in \mathbb{F}_p$ for each $k$. Clearly, it is enough to show that $c_k=0$, for each $k\in \mathbb{F}_p^*$. 
 Since $r-1=p^n-1$ for some $n\geq 2$,
 Lucas' theorem says that $\binom{r-1}{i}\not\equiv 0\mod p$, for $0\leq i\leq r-1$. As $r-p\geq 2$,  
 equating the coefficients of $X^iY^{r-i}$ in both sides for $2\leq i\leq p$, we get that 
 $\underset{k=1}{\overset{p-1}\sum}c_kk^{i-1}\equiv 0\mod p$.
 The non-vanishing of the Vandermode determinant now shows that $c_k=0$, for all $k\in\mathbb{F}_p^*$.
\end{proof}

\begin{remark}
 Note that the proof does not work for $r=p$, since we need $r-p\geq 2$. Also the lemma  is trivially false for $r=p$, because 
 then  $X_{r-1}\subseteq V_r$ must have dimension $\leq p+1$.
\end{remark}

The next proposition describes the structure of $Q$ for $p\mid r$. Note that if $p<r$ is a multiple of $p$, then for $r^\prime=r-1$, 
we have $\Sigma_p(r^\prime)> p-1$, so we can apply Lemma \ref{Xrprime*}.

\begin{prop}
\label{prop4} 
 For $p\geq 3$, let $r\, (>p)$ be a multiple of $p$ such that $r\equiv 1\mod (p-1)$. 
 \begin{enumerate}
  \item[(i)] If $r=p^n$ with $n\geq 2$, then $X_r^* = X_r^{**}=0$ and $X_{r-1}^*=X_{r-1}^{**}$  has dimension $2$.
  \item[(ii)] If $r$ is not a pure $p$-power, then $X_r^*=X_r^{**}\cong V_{p-2}\otimes D$ and $X_{r-1}^*=X_{r-1}^{**}$  has dimension $p+1$.
  \item[(iii)] In either case, $Q$ is a non-trivial extension of $V_1$ by $V_{p-2}\otimes D$, as a $\Gamma$-module.
 \end{enumerate}
\end{prop}

\begin{proof}
 We have
 \begin{enumerate}
 \item [(i)] 
 By Lemma \ref{Xrprime*},  $\dim X_{r^\prime}^* =1$ and $\dim X_{r^\prime} =p+1$  by \cite[(4.5)]{[G78]}.
 By Lemma \ref{dim4}, $\dim X_{r-1} =p+3$.
 By Lemma \ref{onto}, we get a surjection $\phi:X_{r^\prime}\otimes V_1\twoheadrightarrow X_{r-1}$, with a non-zero kernel 
 of dimension $2(p+1)-(p+3)=p-1$. 
 
 Note that $W:=\dfrac{X_{r-1}}{\phi(X_{r^\prime}^*\otimes V_1)}$ is a quotient of $(X_{r^\prime}/X_{r^\prime}^*)\otimes V_1$, 
 which is $M$-isomorphic to $V_{2p-1}$ by \cite[(4.5), (5.3)]{[G78]}. We have 
 the exact sequence of $M$-modules
 $$0 \rightarrow  X_{r^\prime}^*\otimes V_1\xrightarrow{\phi}   X_{r-1}\rightarrow W \rightarrow  0.$$
 Restricting it to the maximal singular submodules, we get the exact sequence
 $$0\rightarrow X_{r^\prime}^*\otimes V_1\xrightarrow{\phi} X_{r-1}^*\rightarrow W^*,$$ where $W^*$ denotes the largest singular submodule of $W$. 
 By Lemmas \ref{dim4} and \ref{Xrprime*} (ii), we get $\dim W =(p+3)-2=p+1$.
 Being a $(p+1)$-dimensional quotient of $V_{2p-1}$, $W$ must be $M$-isomorphic to $V_{2p-1}/V_{2p-1}^*$. 

 By \cite[(4.6)]{[G78]}, $W$ has a unique non-zero minimal submodule, namely, $$W^\prime=\left(X_{2p-1}+V_{2p-1}^*\right)/V_{2p-1}^*.$$ 
 Note that the singular matrix $\left(\begin{smallmatrix}                                                          
                                        1 & 0 \\
                                        0 & 0
                                      \end{smallmatrix} \right)$
 acts trivially on  $X^{2p-1}$, which is non-zero in 
 $W^\prime$. Thus the unique minimal submodule $W^\prime$ is non-singular, so  $W^*=0$, giving us an $M$-isomorphism 
 $X_{r^\prime}^*\otimes V_1\xrightarrow{\phi} X_{r-1}^*.$
 Now by Lemma \ref{Xrprime*} (ii),
 $X_{r-1}^*=\phi(X_{r^\prime}^*\otimes V_1)=X_{r-1}^{**}$ has dimension 2.
              
 \item[(ii)] If $r=p^nu$ for some $n\geq 1$ and $p\nmid u\geq 2p-1$, then  $\dim X_{r-1} =2p+2$, by Lemma \ref{dim2}. 
 We have shown in the proof of Lemma \ref{star elt} (ii) that $X_r^*=X_r^{**}\cong X_u^*$, which is isomorphic to $V_{p-2}\otimes D$, as 
 $p\nmid u$ (cf. Propositions \ref{prop1} and \ref{prop2}).
 We proceed exactly as in the proof of Proposition \ref{prop2}, to get that
 $X_{r-1}^*\cong \phi(X_{r^\prime}^*\otimes V_1)\oplus X_r^*$ has dimension $p+1$. 
 By Lemma \ref{Xrprime*}, we know $\phi(X_{r^\prime}^*\otimes V_1)\subseteq V_r^{**}$. 
 Thus both the summands of $X_{r-1}^*$ are contained in $V_r^{**}$. Hence $X_{r-1}^{**}:=X_{r-1}^*\cap V_r^{**}=X_{r-1}^*$.
 
 \item[(iii)] Using part (i), (ii) above and Lemmas \ref{dim2}, \ref{dim4}, we count that $\dim (X_{r-1}/X_{r-1}^{**}) =p+1$. 
 Hence $\dim Q =2p+2 - \dim X_{r-1} + \dim X_{r-1}^{**} = p+1$.
 Since $X_{r-1}^*=X_{r-1}^{**}$, the natural map $V_r^*/V_r^{**}\rightarrow Q$ 
 is injective, 
 hence an isomorphism by dimension count.
 Now the $\Gamma$-module structure of $Q$ follows from the short exact sequence \eqref{ses**}.
 \end{enumerate}
\end{proof}

Note that in the course of studying the structure of $Q$, we have derived the complete structure of the $M$-submodule 
$X_{r-1}\subseteq V_r$ for $r\equiv 1\mod (p-1)$, summarized as follows:
\begin{prop}\label{JHfactors1} Let $p\geq 3$, $r>p$, and $r\equiv 1\mod (p-1)$. 
 \begin{enumerate}
 \item[(i)] If $\Sigma_p(r-1)=p-1$ (so $p\nmid r$), then $X_{r-1}\cong V_{2p-1}$ as $M$-module.
 \item[(ii)] If $r\neq p^n$ and $\Sigma_p(r-1)>p-1$, then we have a short exact sequence of $M$-modules
                    $$0\rightarrow V_1\otimes D^{p-1}\rightarrow X_{r-1}\rightarrow V_{2p-1}\rightarrow 0.$$
 \item[(iii)] If $r=p^n$ for some $n\geq 2$ (so $\Sigma_p(r-1)>p-1$), then we have a short exact sequence of $M$-modules
                    $$0\rightarrow V_1\otimes D^{p-1}\rightarrow X_{r-1}\rightarrow W\rightarrow 0,$$
 where $W\cong V_{2p-1}/V_{2p-1}^*$ is a non-trivial extension of $V_{p-2}\otimes D$ by $V_1$.
 \end{enumerate}
\end{prop}


\section{Structure of $X_{r-1}$}
\label{structure}

In this section we will study the  $M$-module $X_{r-1}$, i.e., the submodule generated by $X^{r-1}Y$ inside $V_r$, for $r$ 
lying in any congruence class $a$ modulo $(p-1)$.
Recall that for $r\leq p-1$, the module $V_r$ is irreducible, therefore $X_r=X_{r-1}=V_r$. 
We begin with the simple lemma showing that outside this small range of $r$, the module $X_r$ is properly contained in $X_{r-1}$.

\begin{lemma}
\label{Xr n Xr-1}
 For any $p\geq 2$, if $r\geq p$, then $X_r\subsetneq X_{r-1}$.
\end{lemma}

\begin{proof} 
 The inclusion follows from the fact that $X^r=\left(\begin{smallmatrix}
                                                          1 & 1\\
                                                           0 & 1
                                                   \end{smallmatrix}\right)\cdot X^{r-1}Y-X^{r-1}Y$.
 
 If $X_r= X_{r-1}$, then there exist coefficients
 $c_k,d\in\mathbb{F}_p$ such that 
 \begin{equation}
 \label{contra}
 XY^{r-1}=\underset{k\in\mathbb{F}_p}\sum c_k(kX+Y)^r +dX^r.
 \end{equation} 
 Let $r=r_mp^m+\cdots +r_1p+r_0$ be the $p$-adic expansion of $r$. As $r\geq p$, there exists some $j>0$ with $r_j\neq 0$. 
 Equating the coefficients of $XY^{r-1}$ and $X^{p^j}Y^{r-p^j}$ respectively in  both sides of \eqref{contra}, 
 we get $\binom{r}{1}\cdot\underset{k\in\mathbb{F}_p}\sum c_kk=1$, and
 $\binom{r}{p^j}\cdot\underset{k\in\mathbb{F}_p}\sum c_kk^{p^j}=0$. 
 But this implies that $\underset{k\in\mathbb{F}_p}\sum c_kk=0$, as $k^{p^{j}}=k$ for all $k$, and by Lucas' theorem
 $\binom{r}{p^j}\equiv\binom{r_j}{1}\not\equiv 0\mod p$. This is a contradiction.
\end{proof}
 
The following general lemma is valid for $r$ lying in any congruence class $a\mod (p-1)$, with $1\leq a\leq p-1$.  
\begin{lemma}
 \label{dim7}
 Let $p\geq 2$, $r\geq 2p+1$ and $ p\nmid r$. 
 If $\,\Sigma:=\Sigma_p(r-1) \geq p$, then $\dim X_{r-1} =2p+2$.
\end{lemma}
  
\begin{proof}
 We claim that the spanning set $\{X^r, Y^r ,X(kX+Y)^{r-1}, (X+lY)^{r-1}Y : k,l \in\mathbb{F}_p\}$ of $X_{r-1}$ is linearly independent.  
 Suppose there exist constants 
 $A,B,c_k,d_l\in\mathbb{F}_p$ for  $k,l=0,1,\cdots,p-1$, satisfying the equation
 \begin{equation}\label{LI}
  AX^r+BY^r+\underset{k=0}{\overset{p-1}{\sum}}c_kX(kX+Y)^{r-1}+\underset{l=0}{\overset{p-1}\sum}d_l(X+lY)^{r-1}Y=0.
 \end{equation}
 We want to show that $A=B=c_k=d_l=0$, for all $k,l\in\mathbb{F}_p$. It is enough to show that $c_k=d_l=0$, for all $k,l\neq 0$, since that implies that 
 $AX^r+ BY^r+c_0XY^{r-1}+d_0X^{r-1}Y=0$, hence $A=B=c_0=d_0=0$ too. As the matrix $\left( \begin{smallmatrix}
                                                                                        0&1\\
                                                                                        1&0
                                                                                      \end{smallmatrix} \right)$ 
 flips the coefficients $c_k$'s to $d_l$'s in \eqref{LI}, it is enough to show that 
 $d_l=0$ for each $l=1,2,\cdots,p-1$.           
 Let us define for $i,j\geq 0$, $$C_i:=\underset{k=1}{\overset{p-1}\sum}c_kk^i,\: \: D_j=\underset{l=1}{\overset{p-1}\sum}d_ll^j.$$ 
 Note that $C_i,\: D_j$ depends only on
 the congruence classes of $i,j \mod (p-1)$. By the non-vanishing of Vandermode determinant, if  $D_1=D_2=\cdots D_{p-1}=0$, 
 then $(d_1,\cdots,d_{p-1})=(0,\cdots,0)$. Thus the proof reduces to showing that $D_t=0$, for $1\leq t\leq p-1$. 
  
 Comparing the coefficients of $X^{r-t-1}Y^{t+1}$ in both sides of \eqref{LI}, we get 
 \begin{equation}\label{key}
   \binom{r-1}{t+1}C_{r-2-t}+\binom{r-1}{t}D_t=0,\:\text{for } 1\leq t\leq r-3.
 \end{equation}

 Let $r-1=r_mp^m+\cdots+r_1p+r_0$, as in \eqref{C}, with $r_i$ being the rightmost non-zero coefficient. 
 Note that $0\leq r_0< p-1$,
 as by hypothesis $p\nmid r$. 

 We will first consider the $D_t$'s with $1\leq t<r_i$.
 Suppose $r_0=0$, then for $1\leq t< r_i$, choose $t^\prime :=tp^i$. Clearly $r-3\geq t^\prime\equiv t\mod (p-1)$, 
 so we apply it to equation \eqref{key}, 
 to get $D_{t^\prime}=D_t=0$, since $\binom{r-1}{t^\prime+1}\equiv 0\not\equiv\binom{r-1}{t^\prime}\mod p$, by Lucas' theorem.
 If $r_0\neq0$, i.e., $r_i=r_0$, let $i^\prime> 0$ be the smallest positive integer such that
 $r_{i^\prime}\neq 0$. 
 Then we choose $t^\prime:=t+p^{i^\prime}-1\equiv t\mod (p-1)$ and check that $t^\prime\leq r-3$. Now equation \eqref{key} gives us the following two linear equations:
 \begin{eqnarray*}
  \binom{r-1}{t+1}C_{r-2-t} + \binom{r-1}{t}D_t=0,\\
  \binom{r-1}{t^\prime+1}C_{r-2-t} + \binom{r-1}{t^\prime}D_t=0.
 \end{eqnarray*}
 The determinant of the corresponding matrix  is congruent to 
 $$\binom{r_0}{t+1}\binom{r_{i^\prime}}{1}\binom{r_0}{t-1}-\binom{r_0}{t}\binom{r_{i^\prime}}{1}\binom{r_0}{t} \equiv
   -\binom{r_{i^\prime}}{1}\binom{r_0}{t-1}\binom{r_0}{t}\cdot\dfrac{r_0+1}{t(t+1)} \not\equiv 0 \mod p,$$ 
 since we have $0<r_0+1, r_{i^\prime}\leq p-1$. This shows that $D_t=C_{r-2-t}=0$.

 Next we deal with the $D_t$'s, for $r_i\leq t\leq p-1$. As $\Sigma=r_i+r_{i+1}+\cdots+r_m\geq p$, we can write $t$ as $t=r_i+s_{i+1}+\cdots +s_{m}$ with 
 $0\leq s_j\leq r_j$, for all $j$. 
 Let us choose $t^\prime := r_ip^i+s_{i+1}p^{i+1}+\cdots+s_{m}p^{m}$, clearly $t^\prime\equiv t \mod (p-1)$. We observe that since
 $t=r_i+s_{i+1}+\cdots + s_m\leq p-1<p\leq \underset{j=i}{\overset{m}\sum} r_j=\Sigma$, there must exist at least one $j>i$, such that $s_j<r_j$. 
 This implies that $t^\prime\leq r-1-p^j\leq r-1-p^{i+1}\leq r-3$.
 By our choice of $t^\prime$, we have $t^\prime+1\equiv r_0+1\mod p$, with $0\leq r_0< r_0+1\leq p-1$, as $r_0<p-1$. 
 Hence $\binom{r-1}{t^\prime +1}\equiv 0\mod p$ by Lucas' theorem.
 Also note that $\binom{r-1}{t^\prime}\equiv\binom{r_i}{r_i}\binom{r_{i+1}}{s_{i+1}}\cdots\binom{r_m}{s_m}\not\equiv 0\mod p$. Now using \eqref{key}
 for $t^\prime$, we get $D_{t^\prime}=D_t=0$.

 Thus we conclude $D_t=0$ for all $t=1,2,\cdots, p-1$.
\end{proof}

\begin{lemma}
\label{gendim2}
  Let  $p\geq 2$, $ r=p^{n}u$ with $n\geq 1$ and $p\nmid u$, and say $r\equiv a\mod (p-1)$, with $1\leq a\leq p-1$. 
  If $\,\Sigma:=\Sigma_p(u-1)>a-1$, then  $\dim X_{r-1} =2p+2$.
\end{lemma}

\begin{proof} We skip the details of this proof as the basic idea is the same as that of Lemma \ref{dim7}.
\end{proof}

\begin{remark}
 The two lemmas above are valid for all congruence classes $a\mod (p-1)$. Special cases for $a=1$ have been already used in Section \ref{a=1}. 
 Thus the proof of the structure of $X_{r-1}$ for 
 $a=1$, given in Section \ref{a=1}, becomes complete now. Next we will study the structure of $X_{r-1}$, for $r$ lying in higher congruence 
 classes, i.e., for $2\leq a\leq p-1$. The structure of $X_{r-1}$ for $r\equiv 1\mod (p-1)$ will be used as the first step of an inductive process. 
 Note that the condition $2\leq a\leq p-1$ implies that the prime $p$ under consideration is odd. 
\end{remark}

In spite of the lemmas above, the module $X_{r-1}$ often has small dimension, as we are going to show below. 
The next lemma is complementary to Lemma \ref{dim7}, for $a\geq 2$.
Note that the condition $\Sigma_p(r-1)=a-1$ forces $r$ to 
be prime to $p$.

\begin{lemma}
 \label{dim6} 
 Let $p\geq 3$, $r>p$ and  let $r\equiv a\mod (p-1)$ with $2\leq a\leq p-1$.
 If $\,\Sigma=\Sigma_p(r-1)=a-1$, then $\dim X_{r-1} =2a$. In fact,  $X_r\cong V_{a}$  and $X_{r-1}\cong V_{a-2}\otimes D\oplus V_{a}$ as $M$-modules.
\end{lemma}

\begin{proof}
 Write $r^\prime= r-1=r_mp^m+r_{m-1}p^{m-1}+\cdots r_1p+r_0$, where each $r_i\in\{0,1,\cdots, p-1\}$ and $r_0+\cdots +r_m=a-1$, by hypothesis. For any 
 $\alpha,\beta\in\mathbb{F}_p$, we have
 \begin{eqnarray*}
    (\alpha X+\beta Y)^{r^\prime} & = & (\alpha X+\beta Y)^{r_mp^m}\cdots(\alpha X+\beta Y)^{r_1p}\cdot(\alpha X+\beta Y)^{r_0}\\
            & \equiv & (\alpha X^{p^m}+\beta Y^{p^m})^{r_m}\cdots(\alpha X^p+\beta Y^p)^{r_1}\cdot(\alpha X+\beta Y)^{r_0}\mod p.
 \end{eqnarray*}
 Considering this as a polynomial in $\alpha$ and $\beta$, we get that
 $$(\alpha X+\beta Y)^{r^\prime}=\underset{i=0}{\overset{a-1}{\sum}}\alpha^i\beta^{a-1-i}F_i(X,Y),$$
 where the polynomials $F_i$ are independent of $\alpha$ and $\beta$. Hence $X_{r^\prime}$ is contained in the $\mathbb{F}_p$-span of the polynomials 
 $F_0, F_1,\ldots, F_{a-1}$. Thus $\mathrm{dim}(X_{r^\prime})\leq a$. But by \cite[(4.5)]{[G78]}, we know $X_{r^\prime}/X_{r^\prime}^*\cong V_{a-1}$, hence
 $X_{r^\prime}\cong V_{a-1}$ and $X_{r^\prime}^*=0$.
 By \cite[(5.2)]{[G78]} and Lemma \ref{onto}, we have the surjection
 \begin{equation}
 \label{ontomap}
   X_{r^\prime}\otimes V_1\cong V_{a-2}\otimes D\oplus V_{a}\xrightarrow{\phi} X_{r-1}.
 \end{equation}
 By \cite[(4.5)]{[G78]}, we know that $V_{a}$ is a quotient of $X_r$. 
 In fact, by an argument similar to the one just given for $X_{r'}$, we have $X_r\cong V_{a}$ and $X_r^*=0$, since $\Sigma_p(r)=a$.
 Now by Lemma \ref{Xr n Xr-1}, the surjection  in \eqref{ontomap} has to be 
 an isomorphism. 
\end{proof}


\begin{lemma}\label{X_r}
 For any $r>p$, let $r\equiv a\mod (p-1)$, for some $1\leq a\leq p-1$. 
 Then either $X_r^*=0$ or $X_r^*\cong V_{p-a-1}\otimes D^a$ as an $M$-module. 
\end{lemma}

\begin{proof} 
 The statement for $a=1$ is proved in Section \ref{a=1} (cf. Propositions \ref{prop1}, \ref{prop2}, \ref{prop4}). 
 Now we use the method of induction to prove it for $2\leq a\leq p-1$.
 So assuming the statement for all congruence classes less than $a$, we want to prove it for  $r\equiv a \mod (p-1)$.

 Recall that for $r=p^nu$ with $p\nmid u$, $X_{u}\cong X_r$, where the $M$-linear isomorphism is simply given by $F\mapsto F^{p^n}$, for any $F\in X_u$.
 So the largest singular submodules $X_{u}^*$ and $X_r^*$ are $M$-isomorphic. We also note that if 
 $u<p$, then $X_{u}=V_{u}$ and $X_r^*=X_{u}^*=0$. Thus without loss of generality 
 we may assume that $u>p$ and $p\nmid r$, i.e., $r=u>p$.  We denote $r^\prime:=r-1$, as usual. 
 As $p\nmid r$ and $\Sigma\equiv a-1\mod (p-1)$, $r$ satisfies the hypothesis of either Lemma \ref{dim6} or Lemma \ref{dim7}.
 So $\dim X_{r-1}$ is either $2a$ or $2p+2$ respectively.
 In the first case we know $X_r\cong V_{a}$, so $X_r^*=0$ by \cite[(4.5)]{[G78]}. 
 
 In the second case, Lemma \ref{dim7}, Lemma \ref{dim3} and \cite[(4.5)]{[G78]} together tell us that $\dim X_r^*=p-a$. 
 By Lemma \ref{onto}, we have the $M$-map 
 $\phi: X_{r^\prime}\otimes V_1\twoheadrightarrow X_{r-1}$. As $\dim X_{r-1}=2p+2$ in this case, $X_{r^\prime}$ must have  maximum possible dimension 
 $p+1$, and hence $ X_{r^\prime}^*\neq0$. By the induction hypothesis, we get $X_{r^\prime}^*\cong V_{p-a}\otimes D^{a-1}$. Counting dimensions we conclude
 that the map $\phi$ above must be an isomorphism. The short exact sequence
 $$0\longrightarrow X_{r^\prime}^*\otimes V_1\longrightarrow X_{r^\prime}\otimes V_1\longrightarrow V_{a-1}\otimes V_1\longrightarrow 0$$
 implies by \cite[(5.2)]{[G78]}, that we have
 $$0\longrightarrow V_{p-a-1}\otimes D^a\oplus V_{p-a+1}\otimes D^{a-1}\longrightarrow X_{r-1}\longrightarrow V_{a-2}\otimes D\oplus V_{a}\longrightarrow 0,$$
 giving the $M$-module structure of $X_{r-1}$.
 
 The obvious choice for the $(p-a)$-dimensional subspace $X_r^*\subseteq X_{r-1}$ is $V_{p-a-1}\otimes D^a$. In the 
 particular case when $2a=p+1$, we eliminate the possibility of $X_r^*$ being isomorphic to $V_{a-2}\otimes D$ by using the fact that $X_r$ is indecomposable as 
 an $M$-module \cite[(4.6)]{[G78]}, so $X_r$ cannot be $M$-isomorphic to $V_{a-2}\otimes D\oplus V_{a}$.  So $X_r^*\cong V_{p-a-1}\otimes D^a$, as desired.
\end{proof}

Recall that in the case $a=1$, we have $X_r^*/X_r^{**}\neq 0$ if and only if $p\nmid r$,  by Lemma \ref{star elt}. 
Our next lemma shows that for all the higher congruence classes, $X_r^*/X_r^{**}=0$ always.
 
\begin{lemma}
\label{fullcase}
Let  $p\geq 3$, $r\geq 2p$ and $r\equiv a\mod (p-1)$. If $\,2\leq a\leq p-1$, then  $X_r^*=X_r^{**}$.
\end{lemma}

\begin{proof} 
 For $r=2p$, one has $X_r\cong V_2$ and $X_r^*=0$, so there is nothing to prove. So assume $r>2p$. If $X_r^*/X_r^{**}\neq 0$, 
 then by Lemma \ref{X_r} we must have $X_r^*/X_r^{**}\cong V_{p-a-1}\otimes D^a$. But the short exact sequences \eqref{ses**2} and \eqref{ses**3}
 given by Proposition \ref{es2} tell us that $V_{p-a-1}\otimes D^a$ is not a $\Gamma$-submodule of $V_r^*/V_r^{**}$. This is a contradiction.
\end{proof}

The next lemma is complementary to Lemma \ref{gendim2} for $a\geq 2$ and is comparable to  Lemma \ref{dim4} for $a=1$. It generates examples of $X_{r-1}$ with relatively small dimension, 
under the hypothesis $p\mid r$.

\begin{lemma}
\label{dim8}
 Let $p\geq 3$,  $r>2p$ and let $p$ divide $ r\equiv a\mod (p-1)$ with  $2\leq a\leq p-1$. Write $r=p^n u $, where $n\geq 1$ and $p\nmid u$.
 If  $\,\Sigma_p(u-1)=a-1$, then $\dim X_{r-1} =a+p+2$.
\end{lemma}

\begin{proof} 
 Since $p\mid r$, we have $p\nmid r^\prime:=r-1\equiv p-1 \mod p$. As $r>p$, we know that $\Sigma_p(r^\prime-1)$ is at least $ p-1$. Since
 $\Sigma_p(r^\prime-1)\equiv r^\prime-1\equiv a-2\mod (p-1)$, we have $\Sigma_p(r^\prime-1)\geq p+a-3$. If $a\geq 3$, then by Lemma \ref{dim7} and Lemma \ref{dim3} we 
 get $\dim X_{r^\prime} =p+1$.
 On the other hand if $a=2$, then 
 Propositions \ref{prop1} and \ref{prop2} show that $\dim X_{r'}=p+1$.                                     
 In any case, $X_{r^\prime}^*\neq 0$, 
 by \cite[(4.5)]{[G78]}.                                       
 Therefore $X_{r^\prime}^*\cong V_{p-a}\otimes D^{a-1}$, by
 Lemma \ref{X_r}. Now using \cite[(4.5), (5.2)]{[G78]}, we get 
 \begin{displaymath}
\xymatrix{
    0 \ar[r] &  V_{p-a-1}\otimes D^a\oplus V_{p-a+1}\otimes D^{a-1}\ar[r] & X_{r^\prime}\otimes V_1 \ar@{->>}[d]_{\phi} \ar[r] & V_{a-2}\otimes D\oplus V_{a} \ar[r] & 0. \\
     &   & X_{r-1}  &  \\   }
\end{displaymath}

 Since $r=p^nu$, we know that $X_{u}\cong X_r$. 
 As $\Sigma_p(u-1)=a-1$, we have $X_{u}\cong V_{a}$, by Lemma \ref{dim6} if $u>p$, and by the fact that $V_u=V_a$ is irreducible if $u<p$.  So 
 $X_r\cong V_{a}$ has dimension $a+1<p+1$, and Lemma \ref{dim3} implies that $\dim X_{r-1} <2p+2$. 
 Hence all of the four JH factors of $X_{r^\prime}\otimes V_1 $ given  above cannot 
 occur in the quotient $X_{r-1}$. We will show that the last three JH factors, i.e., $V_{p-a+1}\otimes D^{a-1}$, $V_{a-2}\otimes D$, and $V_{a}$ 
 always occur in $X_{r-1}$. Therefore  $V_{p-a-1}\otimes D^a$ must die in $X_{r-1}$. Adding the dimensions of the 
 JH factors we will get $\dim X_{r-1} =a+p+2$, as desired.  

 For $F(X,Y)\in V_{m}$, let $\delta_m(F):=F_X\otimes X+F_Y\otimes Y\in V_{m-1}\otimes V_1$, where $F_X, F_Y$ are the usual partial derivatives of $F$.
 It is shown on \cite[p. 449]{[G78]} that the injection $V_{p-a+1}\otimes D^{a-1}\hookrightarrow (V_{p-a}\otimes D^{a-1})\otimes V_1$
 can be given by the map $\frac{1}{p-a+1}\cdot \delta_{p-a+1}$ ($\bar{\phi}$ in the notation of \cite{[G78]}). 
 So the monomial $X^{p-a+1}\in V_{p-a+1}\otimes D^{a-1}$ maps to 
 the non-zero element $X^{p-a}\otimes X\in(V_{p-a}\otimes D^{a-1})\otimes V_1$.
 Let $F\in X_{r^\prime}^*$ be the image of  $X^{p-a}$ under the isomorphism $V_{p-a}\otimes D^{a-1}\overset{\sim}\rightarrow X_{r^\prime}^*$.
 Then under the composition of maps 
 $$V_{p-a+1}\otimes D^{a-1}\hookrightarrow  (V_{p-a}\otimes D^{a-1})\otimes V_1\overset{\sim}\longrightarrow X_{r^\prime}^*\otimes V_1 \xrightarrow{\phi} X_{r-1},$$ 
 the monomial $X^{p-a+1}\mapsto X^{p-a}\otimes X\mapsto F\otimes X\mapsto X\cdot F\neq 0$ in $X_{r-1}$.
 This shows that $V_{p-a+1}\otimes D^{a-1}$ must be a JH factor of $X_{r-1}$.

 Note that by  hypothesis $r\equiv 0\not\equiv a \mod p$.  We refer to Lemma \ref{mixed} in Section \ref{a=3..} 
 to show that $V_{a-2}\otimes D$ is a JH factor of $X_{r-1}$, if $3\leq a\leq p-1$.
For $a=2$,  the hypothesis implies that $r$ must be of the form $r=p^n(p^m+1)$ with $m+n\geq 2$. An elementary calculation using Lucas' theorem shows that in this case the $M$-linear map 
$X_{r-1}\twoheadrightarrow V_{0}\otimes D$, 
 sending $X^{r-1}Y\mapsto 1$, is  well-defined. Thus $V_{a-2}\otimes D$ is a JH factor of $X_{r-1}$ for all $a\geq 2$. 

Finally, $V_{a}$ is always a JH factor of $X_{r-1}$, by \cite[(4.5)]{[G78]}.
\end{proof}

We write $r=p^nu$, where $p\nmid u$ and $n\geq 0$ is an integer.  
Note that if $p\nmid r$, then $u$ simply equals $r$, and $n=0$. Clearly $u\equiv a\mod (p-1)$ as well, so the sum of the $p$-adic digits of $u-1$  
lies in the congruence class $a-1 \mod (p-1)$. We denote this sum by $\Sigma:=\Sigma_p(u-1)$.
Then the $M$-module structure of $X_{r-1}$ proved in this section can be summarized as follows:
 
\begin{prop}
\label{JHfactors} 
 Let $p\geq 3$, $r>2p$, and $r\equiv a\mod (p-1)$ with $2\leq a\leq p-1$. Then with the notation as  above, 
 \begin{enumerate}
 \item[(i)] If  $\,\Sigma=a-1$ and $p\nmid r$ (i.e., $n=0$), then $\dim X_{r-1} =2a$ and as an $M$-module
 $$X_{r-1}\cong V_{a-2}\otimes D\:\oplus\: V_{a}.$$
 \item[(ii)] If  $\,\Sigma=a-1$ and $p\mid r$ (i.e., $n>0$), then $\dim X_{r-1} =a+p+2$, and 
 we have the following exact sequence of $M$-modules:
 $$0\longrightarrow V_{p-a+1}\otimes D^{a-1}\longrightarrow X_{r-1}\longrightarrow V_{a-2}\otimes D\:\oplus\: V_{a}\rightarrow 0.$$
 \item[(iii)] If  $\,\Sigma>a-1$, or equivalently $\Sigma\geq p+a-2$, then $\dim X_{r-1} =2p+2$, 
 and we have the following exact sequence of $M$-modules:
   $$0\rightarrow V_{p-a-1}\otimes D^a\:\oplus\:V_{p-a+1}\otimes D^{a-1}\rightarrow X_{r-1}\rightarrow V_{a-2}\otimes D\:\oplus\: V_{a}\rightarrow 0.$$
 \end{enumerate}
\end{prop}

\section{$r \equiv 2 \mod (p-1)$}
\label{a=2}

With the notation of Section \ref{basics}, let $a=2$. In particular, $p$ will denote an odd prime throughout this section. 
By Proposition \ref{es2}, we get  $V_r^*/V_r^{**}\cong V_{p-1}\otimes D \oplus V_0\otimes D$ as a $\Gamma$-module. 
For $r^\prime:=r-1$, we know by \cite[(4.5)]{[G78]} that $X_{r^\prime}/X_{r^\prime}^*\cong V_1.$ 
As usual by Lemma \ref{onto} and \cite[(5.2)]{[G78]}, we get the following commutative
diagram of $M$-modules:
\begin{eqnarray}
\label{smalldiagram}
\xymatrix{
    0 \ar[r] &  X_{r^\prime}^*\otimes V_1\ar@{->>}[d]\ar[r] & X_{r^\prime}\otimes V_1 \ar@{->>}[d]_{\phi} \ar[r] & V_{0}\otimes D\oplus V_{2} \ar@{->>}[d]\ar[r] & 0 \\
    0 \ar[r] & \phi(X_{r^\prime}^*\otimes V_1)\ar[r] & X_{r-1} \ar[r] & \dfrac{X_{r-1}}{\phi(X_{r^\prime}^*\otimes V_1)} \ar[r] & 0. \\
       }
\end{eqnarray}

\begin{lemma} 
\label{lem1}
 Let  $p\geq 3$, $r>2p$ and  $r\equiv 2 \mod(p-1)$. If $p\nmid r$, then there is a $M$-linear surjection $$X_{r-1}^*/X_{r-1}^{**}\twoheadrightarrow V_0\otimes D.$$
\end{lemma}

\begin{proof}
 We will show  that the following  composition of maps is non-zero, hence onto:
 $$\dfrac{X_{r-1}^*}{X_{r-1}^{**}}\hookrightarrow\dfrac{V_r^*}{V_r^{**}}
 \cong\dfrac{V_{r-p-1}}{V_{r-p-1}^*}\otimes D\xrightarrow{\psi^{-1}\otimes\,\mathrm{id}}\dfrac{V_{2p-2}}{V_{2p-2}^*}\otimes D\twoheadrightarrow V_0\otimes D,$$
 where $\psi$ is the $M$-isomorphism in \cite[(4.2)]{[G78]} and the rightmost surjection is induced from the $\Gamma$-linear map in \cite[5.3(ii)]{[Br03]}. Note that the composition is automatically $M$-linear, 
 since both its domain and range are singular modules.
 
 Consider $F(X,Y)=X^{r-1}Y-XY^{r-1}=\theta(X,Y)\cdot F_1(X,Y)\in X_{r-1}^*$, where $$F_1(X,Y)=\underset{i=0}{\overset{(r-p-1)/(p-1)}{\sum}} X^{r-p-1-i(p-1)}Y^{i(p-1)}.$$
 Then the image of $F\in V_r^*/V_r^{**}$ in $\dfrac{V_{r-p-1}}{V_{r-p-1}^*}\otimes D$ is $F_1\otimes 1$.
 Further, 
 \begin{eqnarray*}
  \psi^{-1}(F_1)&=&\psi^{-1}(X^{r-p-1})+\left(\frac{r-p-1}{p-1}-1\right)\cdot\psi^{-1}(X^{r-2p}Y^{p-1})+\psi^{-1}(Y^{r-p-1})\\
  &=& X^{2p-2}+\frac{r-2p}{p-1}\cdot X^{p-1}Y^{p-1}+Y^{2p-2}=F_2,\text{ say, }  
  \end{eqnarray*} as an element of $ V_{2p-2}/V_{2p-2}^*$.
 Next we compute the image of $F_2\otimes 1$ in $V_0\otimes D$ under the map given in \cite[5.3(ii)]{[Br03]}, which is 
 $0+(\frac{r-2p}{p-1})\cdot 1+0=\frac{r-2p}{p-1}\not\equiv 0\mod p$, 
 as  $p\nmid r$. Thus the polynomial $F(X,Y)\in X_{r-1}^*/ X_{r-1}^{**}$ maps to a non-zero element in $V_0\otimes D$, as desired.
\end{proof}

\begin{lemma}\label{lem2}
Let $p\geq 3$, $r>2p$, and $r\equiv 2\mod (p-1)$. If $p\nmid r^\prime:=r-1$, then $\phi(X_{r^\prime}^*\otimes V_1)\subseteq X_{r-1}^*$, but 
 $\,\phi(X_{r^\prime}^*\otimes V_1)\nsubseteq X_{r-1}^{**}$.
\end{lemma}

\begin{proof}
 The inclusion is obvious from the definition of the map $\phi$.   
 For the rest, consider $F(X,Y):=\underset{k=0}{\overset{p-1}{\sum}}(kX+Y)^{r^\prime}\in X_{r^\prime}^*$, as shown in the proof of Lemma \ref{star elt} (i).
 The coefficient of $X^{r-1}Y$ in  $\phi(F\otimes X)=F(X,Y)\cdot X$ is the same as the coefficient of $X^{r^\prime-1}Y$ in $F$, which is congruent to
 $-r^\prime\mod p$.  
 By hypothesis $p\nmid r^\prime$, so  $\theta^2$ cannot divide 
 $\phi(F\otimes X)$. Therefore $\phi(X_{r^\prime}^*\otimes V_1)\nsubseteq X_{r-1}^{**}$.
\end{proof}

Lemmas \ref{lem1} and \ref{lem2} will be  used to study  the $\Gamma$-module $Q$. If $p\nmid r(r-1)$, 
then we can use both the lemmas  simultaneously to prove the following simple structure of $Q$.

\begin{prop}\label{prop3}
 Let $p\geq 3$,  $r>2p$ and $r\equiv 2\mod (p-1)$. 
 If $p\nmid r(r-1)$, then  $Q\cong V_{p-3}\otimes D^2$ as a $\Gamma$-module. 
\end{prop}
\begin{proof}
 Since $p\nmid r(r-1)$, we have $\Sigma=\Sigma_p(r-1)>a-1=1$ with the notation of  Proposition \ref{JHfactors}, 
 as otherwise $\Sigma=1$, which forces $r-1$ to be a $p$-power.
 So by  Proposition \ref{JHfactors} (iii), $\dim X_{r-1} =2p+2$,
 and we have
 \begin{equation}
   0\longrightarrow V_{p-3}\otimes D^2\oplus V_{p-1}\otimes D\longrightarrow X_{r-1}\longrightarrow V_0\otimes D\oplus V_2\longrightarrow 0,
 \end{equation}  
 noting that here $V_{p-3}\otimes D^2\oplus V_{p-1}\otimes D=\phi(X_{r^\prime}^*\otimes V_1)\subseteq X_{r-1}$.
 
 Restricting this short exact sequence of $M$-modules to the largest singular submodules, we get the structure of $X_{r-1}^*$:
 \begin{equation}
 \label{JHfact}
   0\longrightarrow V_{p-3}\otimes D^2\oplus V_{p-1}\otimes D\longrightarrow X_{r-1}^*\longrightarrow V_0\otimes D\longrightarrow 0,
 \end{equation}
 where the last surjection follows from the fact that $V_0\otimes D$ is a JH factor of $X_{r-1}^*$, by Lemma \ref{lem1}.

 Next we want to compute the JH factors of $X_{r-1}^{**}$.  
 By Lemma \ref{dim3} and  \cite[(4.5)]{[G78]}, we have
 $\dim X_r^* =p-2$. 
 By Lemmas \ref{fullcase} and \ref{X_r}, we get that $X_r^{**}=X_r^*\cong V_{p-3}\otimes D^2$, hence this submodule of
 $\phi(X_{r^\prime}^*\otimes V_1)\,(\subseteq X_{r-1}^*)$ is in fact contained in $X_{r-1}^{**}$. 
 By Lemma \ref{lem2}, the other JH factor $V_{p-1}\otimes D$ of $\phi(X_{r^\prime}^*\otimes V_1)$ cannot 
 occur in $X_{r-1}^{**}$.  Since $V_0\otimes D$ occurs only once in $X_{r-1}$,  Lemma \ref{lem1} tells that $V_0\otimes D$ cannot be a JH factor of $X_{r-1}^{**}$.
 Thus we get $X_{r-1}^{**}\cong V_{p-3}\otimes D^2$.

 So now we know all the JH factors of $X_{r-1}$, $X_{r-1}^*$ and $X_{r-1}^{**}$. As $X_{r-1}^*/X_{r-1}^{**}$ has two JH factors $V_{p-1}\otimes D$ and
 $V_0\otimes D$, we have $X_{r-1}^*/X_{r-1}^{**}= V_r^*/V_r^{**}$, and the left module in the bottom row of Diagram \eqref{bigdiadram} vanishes. 
 On the other hand $X_{r-1}/X_{r-1}^*$ has only one JH factor $V_2$, so the short exact sequence \eqref{ses*} of $\Gamma$-modules implies that the rightmost module 
 in the bottom row of Diagram \eqref{bigdiadram} is $V_{p-3}\otimes D^2$. Therefore $Q$ is $\Gamma$-isomorphic to $V_{p-3}\otimes D^2$. 
\end{proof}

Next we will treat the case  $p\mid r(r-1)$. 
Note that if $p\mid (r-1)$, then $p\nmid r$ and we can still use Lemma \ref{lem1}. Similarly if $p\mid r$, then $p\nmid (r-1)$ and we can  use Lemma \ref{lem2}. 

\begin{prop}
\label{pdr1} 
  Let $p\geq 3$, $r>2p$ and $r\equiv 2\mod (p-1)$.
  If  $p\mid r-1$, then  $Q\cong V_{p-1}\otimes D\oplus V_{p-3}\otimes D^2$ as a $\Gamma$-module.
\end{prop}

\begin{proof}
 If $r-1$ is a pure $p$-power, then by Proposition \ref{JHfactors} (i), $X_{r-1}\cong V_0\otimes D\oplus V_2$ as an $M$-module. Hence 
 $X_{r-1}^*\cong V_0\otimes D$, being the largest singular submodule of $X_{r-1}$. By Lemma \ref{lem1}, $X_{r-1}^{**}$ must be zero. In other words,
 $X_{r-1}/X_{r-1}^*\cong V_2$ and $X_{r-1}^*/X_{r-1}^{**}\cong V_0\otimes D$.
 
 If $r-1$ is not a pure $p$-power, then $\Sigma_p(r-1)\geq p$. By Proposition \ref{JHfactors} (iii), $\dim X_{r-1} =2p+2$ 
 and we have the exact sequence of $M$-modules 
 $$0\longrightarrow  V_{p-3}\otimes D^2\oplus V_{p-1}\otimes D\longrightarrow X_{r-1}\longrightarrow V_0\otimes D\oplus V_2\longrightarrow 0.$$
 Note that $ V_{p-3}\otimes D^2\oplus V_{p-1}\otimes D=\phi(X_{r^\prime}^*\otimes V_1)\subseteq X_{r-1}$, with the notation of Diagram \eqref{smalldiagram}.
 Similarly as in the proof of Proposition \ref{prop3}, we use Lemma \ref{lem1} to obtain
 $$0\longrightarrow \phi(X_{r^\prime}^*\otimes V_1)\longrightarrow X_{r-1}^*\longrightarrow V_0\otimes D\longrightarrow 0.$$
 As $p\mid r^\prime$, by Lemma \ref{star elt} (ii) we have $X_{r^\prime}^*=X_{r^\prime}^{**}$, therefore 
 $\phi(X_{r^\prime}^*\otimes V_1)=\phi(X_{r^\prime}^{**}\otimes V_1)$ must be contained inside $X_{r-1}^{**}$. On the other hand $V_0\otimes D$ occurs only 
 once in $X_{r-1}$, so Lemma \ref{lem1} implies that $V_0\otimes D$ cannot be a JH factor of $X_{r-1}^{**}$. 
 Thus $X_{r-1}^{**}=\phi(X_{r^\prime}^*\otimes V_1)\cong V_{p-3}\otimes D^2\oplus V_{p-1}\otimes D$. Thus again we get 
 $X_{r-1}/X_{r-1}^*\cong V_2$ and $X_{r-1}^*/X_{r-1}^{**}\cong V_0\otimes D$.
 
 Therefore by the  exact sequences \eqref{ses*} and \eqref{ses**2}, the bottom row of Diagram \eqref{bigdiadram} reduces to
 $$ 0\longrightarrow V_{p-1}\otimes D\longrightarrow Q\longrightarrow V_{p-3}\otimes D^2 \longrightarrow 0.$$
 Now the result follows as $V_{p-1}\otimes D$ is an injective $\Gamma$-module. 
\end{proof}

\begin{prop}
\label{pdr2}  
 Let $p\geq 3$,  $r>2p$ and $r\equiv 2\mod (p-1)$.
 If $p\mid r$, then the $\Gamma$-module structure of $Q$ is given by 
 $$0\longrightarrow V_0\otimes D\longrightarrow Q\longrightarrow V_{p-3}\otimes D^2\longrightarrow 0.$$
\end{prop}

\begin{proof}
 Let us write $r=p^nu$, $n\geq 1$ and $p\nmid u$.
 Looking at Diagram \eqref{smalldiagram} and using Proposition \ref{JHfactors}, we get the exact sequence of $M$-modules
 \begin{eqnarray*}0\longrightarrow W\longrightarrow X_{r-1}\longrightarrow V_0\otimes D\oplus V_2\longrightarrow 0,\end{eqnarray*}
 where  $W=\phi(X_{r^\prime}^*\otimes V_1)=\begin{cases} V_{p-1}\otimes D, &\text{if } \Sigma_p(u-1)=1,\\
                                                                  V_{p-1}\otimes D\oplus V_{p-3}\otimes D^2, &\text{if } \Sigma_p(u-1)\geq p.  
                                                                \end{cases}$
                                                                
 Restricting the above sequence to the largest singular submodules, we get
 \begin{equation}
 \label{W} 
   0\longrightarrow W \longrightarrow X_{r-1}^*\longrightarrow V_0\otimes D\longrightarrow 0,
 \end{equation}
 where the last map has to be a surjection, as otherwise $V_0\otimes D$ will be a JH factor of $V_r/V_r^*$, which is not true by the sequence \eqref{ses*}.
 
 Now we claim that the restriction of the  map  $X_{r-1}^*\twoheadrightarrow V_0\otimes D$ in \eqref{W} to the submodule  $X_{r-1}^{**}$ is also non-zero, 
 and hence surjective.

 \underline{Proof of claim}: If not, then the above map factors like $\eta_1:X_{r-1}^*/X_{r-1}^{**}\twoheadrightarrow V_0\otimes D\subseteq V_1\otimes V_1$.
 We can check that the elements $X^{r^\prime},Y^{r^\prime}\in X_{r^\prime}$ map to $X,Y\in V_1$ respectively under the map
 $X_{r^\prime}\twoheadrightarrow X_{r^\prime}/X_{r^\prime}^*\cong V_1$. Hence for 
 $F(X,Y)=X^{r-1}Y-Y^{r-1}X=\phi(X^{r^\prime}\otimes Y- Y^{r^\prime}\otimes X)\in \phi(X_{r^\prime}\otimes V_1)$, we get  
 $\eta_1(F)=X\otimes Y-Y\otimes X\in V_1\otimes V_1$. Note that the non-zero element $X\otimes Y-Y\otimes X$ of 
 $V_1\otimes V_1\cong V_0\otimes D\oplus V_2$ actually belongs to
 $V_0\otimes D$, as it projects to  $XY-YX=0\in V_2$. So $\eta_1(F)\neq 0$.
 Now we consider the composition of maps $\eta_2:X_{r-1}^*/X_{r-1}^{**}\hookrightarrow V_r^*/V_r^{**}\twoheadrightarrow V_0\otimes D$, as described in the proof of Lemma \ref{lem1}. 
 If $\eta_2$  is the zero map, then the 
 short exact sequence \eqref{ses**} implies that $X_{r-1}^*/X_{r-1}^{**}\subseteq V_{p-1}\otimes D$, contradicting the existence of
 $\eta_1:X_{r-1}^*/X_{r-1}^{**}\twoheadrightarrow V_0\otimes D$. So $\eta_2$ is a non-zero map. But the calculation in the proof of 
 Lemma \ref{lem1} shows that $\eta_2(F)=0$, since $p\mid r$. Thus we end up with two different surjections 
 $\eta_1,\eta_2:X_{r-1}^*/X_{r-1}^{**}\twoheadrightarrow V_0\otimes D$,
 one containing $F$ in its kernel and the other not. This forces $V_0\otimes D$ to be a repeated JH factor of $X_{r-1}^*/X_{r-1}^{**}$,  
 contradicting the fact that it occurs only once in $X_{r-1}^*$.  Thus the claim is proved.

 If $\Sigma_p(u-1)\geq p$, then the extra JH factor $V_{p-3}\otimes D^2$ of $W$ is actually the submodule $X_r^*$ (cf. Lemma \ref{X_r}), 
 which equals $X_r^{**}$ by Lemma \ref{fullcase}. 
 So this JH factor is in fact contained in $X_{r-1}^{**}$. By Lemma \ref{lem2}, we know that 
 $W=\phi(X_{r^\prime}^*\otimes V_1)\nsubseteq X_{r-1}^{**}$. 
 Therefore    $$W_1:=W\cap X_{r-1}^{**}=\begin{cases}
                                                          0, &\text{ if } \Sigma_p(u-1)=1,\\
                                                          V_{p-3}\otimes D^2, &\text{ if } \Sigma_p(u-1)\geq p.
                                                         \end{cases}$$
 By the claim above, the structure of $X_{r-1}^{**}$ is now given by
 $$0\rightarrow W_1\rightarrow X_{r-1}^{**}\rightarrow V_0\otimes D\rightarrow 0.$$ 
 Hence we always have $X_{r-1}/X_{r-1}^*\cong V_2$, and $X_{r-1}^*/X_{r-1}^{**}\cong W/W_1\cong V_{p-1}\otimes D$. Next we use  the exact 
 sequences \eqref{ses*} and \eqref{ses**2}, and the bottom row of Diagram \eqref{bigdiadram} reduces to 
 $$0\longrightarrow V_0\otimes D\longrightarrow Q\longrightarrow V_{p-3}\otimes D^2\longrightarrow 0.$$
\end{proof}

\begin{remark}
  The short exact sequences in Proposition \ref{es2} split only when $a=2$, the case we are dealing with. 
 Note that all three possible non-zero submodules of $V_r^*/V_r^{**}\cong V_0\otimes D\oplus V_{p-1}\otimes D$ occur as $X_{r-1}^*/X_{r-1}^{**}$ 
 in the three distinct cases $p\nmid r(r-1)$, $p\mid (r-1)$ and $p\mid r$. 
\end{remark}

\section{$r\equiv a\mod (p-1)$, with $3\leq a\leq p-1$}
\label{a=3..}

In this section we will describe the $\Gamma$-module structure of $Q$ for $r\equiv a\mod (p-1)$ with $3\leq a\leq p-1$. 
Note that this bound on $a$ forces the prime $p$ to be at least $5$.

\begin{lemma}
\label{mixed}
 Let $p\geq 5$, $r>2p$, $r\equiv a\mod (p-1)$ with $3\leq a\leq p-1$. If $r\not\equiv a\mod p$, then
 $X_{r-1}^{*}/ X_{r-1}^{**}$ contains  $V_{a-2}\otimes D$ as a $\Gamma$-submodule. 
\end{lemma}

\begin{proof} 
 By the non-split short exact sequence \eqref{ses**3}, it is enough to show that the submodule $X_{r-1}^*/X_{r-1}^{**}$ of $V_r^*/V_r^{**}$ is non-zero.
 Consider the polynomial $$F(X,Y)=(a-1)X^{r-1}Y+\underset{k=0}{\overset{p-1}{\sum}}k^{p+1-a}(kX+Y)^{r-1}X\in X_{r-1}.$$
 We compute that
 $$F(X,Y)\equiv(a-1)X^{r-1}Y - \underset{\substack{0< j< r-1,\\ j\equiv a-2\mod (p-1)}}\sum\binom{r-1}{j}X^{j+1}Y^{r-j-1}\mod p.$$
 Clearly the monomials in $F(X,Y)$ have $Y$-degree congruent to $1\mod(p-1)$, i.e., they are scalar multiples of 
 $X^{a-1}Y^{r-a+1},\cdots,\,X^{r-p}Y^p\text{ and } X^{r-1}Y$. By Lemma \ref{comb3a}, we get
 $$\underset{\substack{0< j< r-1,\\ j\equiv a-2\mod (p-1)}}\sum\binom{r-1}{j}\equiv(a-1)-(r-1)+\binom{r-1}{r-2}= a-1\mod p.$$ 
 We apply Lemma \ref{elelemma} (i) to conclude that $F\in V_r^*$. 
 But $F\notin V_r^{**}$ as the coefficient of $X^{r-1}Y$ in $F(X,Y)$ equals $(a-1)-\binom{r-1}{r-2}=a-r\not\equiv 0\mod p$, by hypothesis. 
 So 
  $F$ maps to a non-zero element in $X_{r-1}^*/X_{r-1}^{**}$.
\end{proof}

\begin{lemma}
\label{exception}
 Let $p\geq 5$ and $r\equiv a\mod (p-1)$ with $3\leq a\leq p-1$. If $r\equiv a\mod p$, then we have $X_{r-1}\subseteq X_r+V_r^{**}$ and $X_{r-1}^*/X_{r-1}^{**}=0$.
\end{lemma}

\begin{proof} The lemma is trivial for $r=a$ as $V_a$ is irreducible. Thus for each $a$, $r=p^2-p+a$ is the first non-trivial integer satisfying the hypotheses.
 We begin by expanding the following element of $X_r$: 
 $$\underset{k\in\mathbb{F}_p}\sum k^{p-2}(kX+Y)^r\equiv -rXY^{r-1}-\underset{\substack{0<j<r-1\\j\equiv a-1\mod (p-1)}}\sum\binom{r}{j}\cdot X^{r-j}Y^j\mod p,$$ 
 Next we claim that $F(X,Y):=\underset{\substack{0<j<r-1\\j\equiv a-1\mod (p-1)}}\sum\binom{r}{j}\cdot X^{r-j}Y^j$ is in fact contained in $V_r^{**}$,  implying that the monomial $-rXY^{r-1}$ is contained in $X_r+V_r^{**}$.
Since $r\equiv a\not\equiv 0\mod p$, we will have $X_{r-1}\subseteq X_r+V_r^{**}.$

\underline{Proof of the claim}: By Lemma \ref{comb3a}, the sum of the coefficients of $F(X,Y)$ is
$$\underset{\substack{0<j<r-1\\j\equiv a-1\mod (p-1)}}\sum\binom{r}{j}\equiv a-r\equiv 0\mod p,$$
by hypothesis.
Using the same lemma for $r-1$ (note that $2\leq a-1\leq p-2$), we get  
$$\underset{\substack{0<j<r-1\\j\equiv a-1\mod (p-1)}}\sum j\binom{r}{j}=\underset{\substack{0<j<r-1\\j\equiv a-1\mod (p-1)}}\sum r\binom{r-1}{j-1}\equiv r\cdot\left((a-1)-(r-1)\right)\equiv 0\mod p.$$
Then we apply Proposition \ref{elelemma} (ii) to prove the claim.

Thus we have showed $X_{r-1}\subseteq X_r+V_r^{**}$. Hence $X_{r-1}^*\subseteq X_{r-1}\cap(X_r^*+V_r^{**})$, which equals $X_{r-1}\cap(X_r^{**}+V_r^{**})$ by Lemma \ref{fullcase}. 
But $X_r^{**}\subseteq V_r^{**}$, so we have $X_{r-1}^*\subseteq X_{r-1}\cap V_r^{**}=X_{r-1}^{**} $. The reverse inclusion is obvious and therefore $X_{r-1}^*=X_{r-1}^{**}$.
 \end{proof}

\begin{remark}
  It was proved in \cite[Prop. 4]{[GG15]} that $X_{r-1} \not\subset X_r + V_r^{**}$, for $p \leq r \leq p^2-p+2$ and all primes $p$. 
  The lemma above shows that the upper bound is sharp, at least if $p \geq 5$, since the smallest non-trivial $r$ covered by 
  the lemma is $p^2-p+3$. As we show just below, the fact that $X_{r-1} \subset X_r + V_r^{**}$ causes $Q$ to have $3$ JH factors,
  leading to additional complications. 
\end{remark}

\begin{prop}\label{prop5}
 Let $p\geq 5$, $r>2p$ and $r\equiv a\mod (p-1)$ with $3\leq a\leq p-1$. The $\Gamma$-module structure of $Q$ is as follows.
 \begin{enumerate}
 \item[(i)] If $r\not\equiv a\mod p$, then
 $$0\longrightarrow V_{p-a+1}\otimes D^{a-1}\longrightarrow Q\longrightarrow V_{p-a-1}\otimes D^a\longrightarrow 0,$$
 and moreover the exact sequence above  is $\Gamma$-split.
 \item[(ii)] If  $r\equiv a\mod p$, then 
 $$0\longrightarrow V_r^*/V_r^{**}\longrightarrow Q\longrightarrow V_{p-a-1}\otimes D^a\longrightarrow 0,$$
 where $V_r^*/V_r^{**}$ is the non-trivial extension of  $V_{p-a+1}\otimes D^{a-1}$ by $V_{a-2}\otimes D$.  \end{enumerate}
\end{prop}

\begin{proof} 
 Let $r^\prime:=r-1$. As explained in  Section \ref{structure} (cf. Proposition \ref{JHfactors}), we have 
 $$0\longrightarrow \phi(X_{r^\prime}^*\otimes V_1)\longrightarrow X_{r-1}\longrightarrow  V_{a-2}\otimes D\oplus V_{a}\rightarrow 0.$$
 Restricting the above $M$-linear maps to the largest singular submodules, we get
 $$0\longrightarrow \phi(X_{r^\prime}^*\otimes V_1)\longrightarrow X_{r-1}^*\longrightarrow V_{a-2}\otimes D\rightarrow 0.$$
 Indeed, the short exact sequence \eqref{ses*} shows that $V_{a-2}\otimes D$ is not  a JH factor of $V_r/V_r^*$, 
 so the surjection $X_{r-1}\twoheadrightarrow V_{a-2}\otimes D$ cannot factor through
 $X_{r-1}/X_{r-1}^*$, hence the rightmost map above is surjective.
 As $2\leq a-1\leq p-2$, by  Lemma \ref{fullcase} we have $X_{r^\prime}^*=X_{r^\prime}^{**}$. Hence by the definition of the $\phi$ map, 
 we get $\phi(X_{r^\prime}^*\otimes V_1)\subseteq X_{r-1}^{**}$. 

 If $r\not\equiv a\mod p$, then Lemma \ref{mixed}  implies that 
 $X_{r-1}^*/X_{r-1}^{**}$ must have exactly one JH factor, namely $V_{a-2}\otimes D$. 
 So we have $X_{r-1}/X_{r-1}^{*}\cong V_{a}$ and $X_{r-1}^*/X_{r-1}^{**}\cong V_{a-2}\otimes D$. Now  using the short exact sequences 
 \eqref{ses*} and \eqref{ses**3},  the bottom row of Diagram \eqref{bigdiadram} reduces to
 $$0\longrightarrow V_{p-a+1}\otimes D^{a-1}\longrightarrow Q\longrightarrow V_{p-a-1}\otimes D^a\longrightarrow 0.$$
 That the sequence splits follows from \cite[Cor. 5.6 (i)]{[BP]}, which implies that there exists no non-trivial extension between the $
 \Gamma$-modules $V_{p-a-1}\otimes D^a$ and $V_{p-a+1}\otimes D^{a-1}$. 

 On the other hand if $r\equiv a\mod p$, then $X_{r-1}/X_{r-1}^{*}\cong V_{a}$ as before, but  $X_{r-1}^*/X_{r-1}^{**}=0$, by Lemma \ref{exception}. 
 Now using the short exact sequence 
 \eqref{ses*},  the bottom row of Diagram \eqref{bigdiadram} reduces to
 $$0\longrightarrow V_r^*/V_r^{**}\longrightarrow Q\longrightarrow V_{p-a-1}\otimes D^a\longrightarrow 0,$$
 where the structure of $V_r^*/V_r^{**}$ is given by the exact sequence \eqref{ses**3}.
\end{proof}



\section{Combinatorial lemmas}
\label{combinatorial}

In this section we prove some technical lemmas which are used repeatedly in the next two sections. 
Only the first lemma is proved in detail as the techniques used to prove the others are similar.

\begin{lemma}
\label{comb2}
 Let $r\equiv a\mod(p-1)$ with $2\leq a\leq p-1$. Then one can choose integers $\alpha_j\in\mathbb{Z}$, 
 for all $j$ with $0<j<r$ and $j\equiv a\mod (p-1)$, such that 
 \begin{enumerate}
   \item[(i)] $\binom{r}{j}\equiv \alpha_j \mod p$, for all $j$ as above,
  
   \item[(ii)] $\underset{j}\sum \alpha_j\equiv 0\mod p^3$,
  
   \item[(iii)] $\underset{j\geq 1}\sum j\,\alpha_j\equiv 0\mod p^2$,
  
   \item[(iv)] $\underset{j\geq 2}\sum\binom{j}{2}\,\alpha_j \equiv \begin{cases}\binom{r}{2}\mod p, &\text{if } a=2,\\
                                                                    0\mod p, & \text{if } 3\leq a\leq p-1. 
                                                       \end{cases}$
 \end{enumerate}
\end{lemma}

\begin{proof}
  For $r \leq ap$,  note that $\Sigma_p(r)=a$ and one can check using Lucas' theorem that $\binom{r}{j}\equiv 0\mod p$ for all the $j$'s listed above. 
  In this case we simply choose
  $\alpha_j=0$, for all $j$. So now assume $r>ap$, hence $j=a, ap$ are both contained in the list of $j$'s above. 
  Let  $a^\prime$ be a fixed  integer such that $a^\prime a
  \equiv 1\mod p^2$, and then let us choose $\alpha_j$'s for all $0<j<r$ with $j\equiv a\mod(p-1)$ as follows:
 
  \begin{eqnarray}
    \alpha_j &=& \binom{r}{j}, \text{ for}\:j\neq a, ap,\\
    \label{alpha_a}\alpha_a &=& - \sum_{\substack{a\,<\,j\,<\,r,\\ j\equiv a\mod(p-1)}} a^\prime j\,\binom{r}{j},\\
    \alpha_{ap} &=& 
    -\underset{j\neq a,\:ap}\sum\binom{r}{j}-\alpha_a.
  \end{eqnarray}
  We will show that this choice of $\alpha_j$'s satisfy  the properties {(i)}, {(ii)}, {(iii)}, and {(iv)}.
  \begin{enumerate}
  \item[(i)]: Note that $j\,\binom{r}{j}=r\,\binom{r-1}{j-1}$, for any $j\geq 1$. Using Lemma \ref{comb1} for $r-1$ and $r$ respectively, we obtain
  \begin{eqnarray*}\alpha_a &=& -\sum_{\substack{a<j<r,\\ j\equiv a\mod(p-1)}} a^\prime r\,\binom{r-1}{j-1}\:\overset{(\ref{comb1})}
  \equiv \:a^\prime r\,\binom{r-1}{a-1}= a^\prime a\,\binom{r}{a}\equiv\binom{r}{a}\mod p,\\
  \alpha_{ap} &{\overset{(\ref{comb1})}\equiv} & \binom{r}{a}+\binom{r}{ap}-\alpha_a\equiv\binom{r}{ap}\mod p.\end{eqnarray*}
  For $j\neq a,\, ap$, property (i) is trivially satisfied.
 
  \item[(ii)]: By  our choice of $\alpha_j$'s, note that in fact $\underset{j}\sum\alpha_j=0$.
 
  \item[(iii)]: Since $a^\prime a\equiv 1\mod p^2$, using equation \eqref{alpha_a} we get   
  $a\cdot\alpha_a\equiv- \underset{\substack{a<j<r,\\ j\equiv a\mod(p-1)}}\sum j\,\binom{r}{j}\mod p^2$.
  Since $\alpha_{ap}\equiv\binom{r}{ap}\mod p$, we have
  $ap\cdot\alpha_{ap}\equiv ap\binom{r}{ap}\mod p^2.$
  Using these two congruences we  conclude that  $\underset{j}\sum\, j\cdot\alpha_j\equiv 0\mod p^2$,  as desired.
 
  \item[(iv)]: We use  property (i) and Lemma \ref{comb1} for $r-2$ to get 
  \begin{eqnarray*}\underset{\substack{0<j<r,\\j\equiv a\mod (p-1)}}\sum\binom{j}{2}\alpha_j
                   & \equiv & \underset{\substack{0<j<r,\\j\equiv a\mod (p-1)}}\sum\binom{j}{2}\binom{r}{j}=
                            \underset{\substack{0<j<r,\\j\equiv a\mod (p-1)}}\sum\binom{r}{2}\cdot\binom{r-2}{j-2}\\
                   & \equiv &  \begin{cases}
                                  \binom{r}{2}\mod p, &\text{if } a=2,\\
                                             0\mod p, & \text{if } 3\leq a\leq p-1. 
                               \end{cases}
  \end{eqnarray*}
  \end{enumerate}
\end{proof}

\begin{lemma}\label{comb3}
 Let $r\equiv b\mod (p-1)$, 
 and $ 3\leq b\leq p$. If $p\mid (b-r)$, then one can choose  integers $\beta_j$,  for all 
 $j\equiv b-1\mod p-1 \text{ with } b-1\leq j< r-1$, satisfying:
 \begin{enumerate}
  \item[(1)] $\beta_j\equiv\binom{r}{j} \mod p$ for all $j$ as above,
  \item[(2)] $\underset{j\geq n}\sum \,\binom{j}{n}\,\beta_j\equiv 0\mod p^{3-n}$, for $n=0,1,\text{ and } 2$.
 \end{enumerate}
\end{lemma}
\begin{proof}
 %
 As $p\mid (b-r)$, we have $r\equiv b\mod (p^2-p)$, so we may assume $r\geq p^2-p+b$. Thus we have $j=b-1, (b-1)p$ are two of the $j$'s in the 
 expression for $T_r$ in Lemma \ref{comb3a}. Let us choose
 \begin{eqnarray}
  \beta_j &=&\binom{r}{j}, \:\text{ for all } j\neq b-1, (b-1)p,\\
  \beta_{b-1} &=&-\sum_{\substack{b-1\,<j\,<\,r-1,\\ j\equiv b-1\mod (p-1)}} b^\prime j\binom{r}{j},\\
  \beta_{(b-1)p} &=&-\sum_{\substack{j\neq b-1,\\j\neq (b-1)p}}\binom{r}{j}-\beta_{b-1},
 \end{eqnarray} where $b^\prime$ is any integer satisfying $(b-1)b^\prime\equiv 1\mod p^2$. 
 
 One can now check that this choice of the integers $\beta_j$ satisfy the required 
 properties. The proof uses the congruence in Lemma \ref{comb3a} and is similar to that of Lemma \ref{comb2}, so we leave it as an exercise for 
 the reader. 
\end{proof}

\begin{lemma}\label{comb5}
 Let $p\geq 3$,  $r\equiv 1\mod (p-1)$, i.e., $b=p$ and let $p^2\mid (p-r)$. Then
 \begin{enumerate}\item[(i)] One can choose 
 $\alpha_j\,\in\mathbb{Z}$, for all $j\equiv 1\mod (p-1)$ with $p\leq j< r$, satisfying:
 
 {\em (1)} $\alpha_j\equiv\binom{r}{j} \mod p^2$, for all $j$ as above,
  
 {\em (2)} $\underset{j\geq n}\sum \,\binom{j}{n}\alpha_j\equiv 0\mod p^{4-n}$, for $n=0,1,\,\text{ and } 2$,
 
 {\em (3)} $\underset{j\geq 3}\sum\binom{j}{3}\alpha_j\equiv\begin{cases}
                                                          0\mod p, & \text{if } p\geq 5\\
                                                          1\mod p, & \text{if } p=3.
                                                         \end{cases}$

\item[(ii)]  One can choose 
 $\gamma_j\,\in\mathbb{Z}$, for all $j\equiv 0\mod (p-1)$ with $p-1\leq j< r-1$, satisfying:
 
 {\em (1)} $\gamma_j\equiv\binom{r}{j} \mod p^2$, for all $j$ as above,
  
 {\em (2)} $\underset{j\geq n}\sum \,\binom{j}{n}\gamma_j\equiv 0\mod p^{4-n}$, for $n=0,1,\,\text{ and }2$,
 
 {\em (3)} $\underset{j\geq 3}\sum\binom{j}{3}\gamma_j\equiv\begin{cases}
                                                          0\mod p, & \text{if } p\geq 5\\
                                                          -1\mod p, & \text{if } p=3.
                                                         \end{cases}$
                                                         \end{enumerate}
\end{lemma}

\begin{proof} 
The proof is  similar to that of Lemma \ref{comb2} or  \ref{comb3}, once we use the congruences given in Lemma \ref{comb4}.
\end{proof}

\begin{remark}
 The integers $\alpha_j$, $\beta_j$, $\gamma_j$'s in the lemmas above are not unique, but the existence of such integers is crucial. We will use them in \S \ref{elimination} and \S \ref{sep} to construct
 functions to eliminate JH factors of $Q$ and to compute the reduction $\bar V_{k,a_p}$, which is the main goal of this paper.
\end{remark}

\section{Elimination of JH factors}\label{elimination}

For the rest of this paper, we will work under the assumption that $1<v(a_p)<2$.

Let us recall  Proposition 3.3 in \cite{[BG09]}: If $\bar{\Theta}_{k,a_p}$ is a quotient of $\mathrm{ind}_{KZ}^{G}(V_s\otimes D^n)$ for some $0\leq s\leq p-1$, then 
$\bar{V}_{k,a_p}\cong \mathrm{ind}(\omega_2^{s+1+(p+1)n})$, unless $s=p-2$,  where one has the additional possibility that $\bar{V}_{k,a_p}$ is reducible and 
isomorphic to $\omega^n\oplus\omega^n$ on $I_p$, the inertia subgroup at $p$. Using this result one can specify the shape of $\bar{V}_{k,a_p}$ when $Q$ 
is irreducible as a $\Gamma$-module, up to the fact that $\bar{V}_{k,a_p}$ may be occasionally reducible, as mentioned above. 
For example, we have 

\begin{theorem}
 \label{thm1}
 Let $p\geq 3$ and $r>2p$. 
 \begin{enumerate}
 \item [(i)]  If $r\equiv 1\mod (p-1)$ and $p\nmid r$, then $\bar{V}_{k,a_p}\cong \mathrm{ind}(\omega_2^2)$. 
 If $p=3$, then  $\bar V_{k,a_p}$ can also possibly be reducible and trivial on $I_p$. 
 \item [(ii)] If $r\equiv 2\mod (p-1)$ and $p\nmid r(r-1)$, then $\bar{V}_{k,a_p}\cong \mathrm{ind}(\omega_2^3)$.
 \end{enumerate}
\end{theorem}

\begin{proof} 
 As $1<v(a_p)<2$, there exists a surjection $\mathrm{ind}_{KZ}^G Q\twoheadrightarrow\bar{\Theta}_{k,a_p}$. 
 In part (i), we have $Q\cong V_1$ by  Propositions \ref{prop1} and \ref{prop2}.
 In part (ii), $Q\cong V_{p-3}\otimes D^2$ by Proposition \ref{prop3}. 
 Since $Q$ is irreducible in both cases, we apply \cite[Prop. 3.3]{[BG09]} to determine $\bar{V}_{k,a_p}$.
\end{proof}

\begin{remark}
 In fact, Theorem \ref{reducible1} in the next section  will imply that the reducible possibility in part (i) above never occurs. Note that for $p=3$, the condition $p\nmid r$
 implies that $\binom{r-1}{2}\equiv 0\mod p$  and therefore the hypothesis of
 Theorem \ref{reducible1} is automatically satisfied.
\end{remark}

As we have already seen, $Q$ is usually not irreducible. It can have two and sometimes even three JH factors as a $\Gamma$-module depending on the 
congruence class of $r$ modulo both $p-1$ and $p$. In these cases we will use the 
explicit formula for the Hecke operator $T$ acting on the space $\mathrm{ind}_{KZ}^G\mathrm{Sym}^r\bar{\mathbb{Q}}_p^2$, to eliminate one or 
two JH factors of $Q$, so that we can use \cite[Prop. 3.3]{[BG09]}. 

To work explicitly with the Hecke operator $T $, we need to recall some
well-known formulas involving $T$ from \cite{[Br03]}. For $m = 0$, set $I_0 = \{0\}$, and 
for $m >0$, let
$$I_m = \{ [\lambda_0] + [\lambda_1] p + \cdots + [\lambda_{m-1}]p^{m-1}  \> : \>  \lambda_i \in \F_p \} 
              \subset \Z_p,$$
where the square brackets denote Teichm\"uller representatives. For $m \geq 1$, 
there is a truncation map
$[\quad]_{m-1}: I_{m} \rightarrow I_{m-1}$ given by taking the first $m-1$ terms in the $p$-adic expansion above;
for $m = 1$, $[\quad]_{m-1}$ is the $0$-map.
Let $\alpha =  \left( \begin{smallmatrix} 1 & 0 \\ 0  & p \end{smallmatrix} \right)$. 
For $m \geq 0$ and $\lambda \in I_m$, let
\begin{eqnarray*}
  g^0_{m, \lambda} =  \left( \begin{matrix} p^m & \lambda \\ 0 & 1 \end{matrix} \right) & \quad \text{and} \quad 
  g^1_{m, \lambda} = \left( \begin{matrix} 1 & 0  \\ p \lambda  & p^{m+1} \end{matrix} \right),
\end{eqnarray*}
noting that $g^0_{0,0}=\mathrm{Id}$ is the identity matrix and $g_{0,0}^1=\alpha$ in $G$. 
Recall the decomposition
\begin{eqnarray*}
    G & = & \coprod_{\substack{m\geq 0,\,\lambda \in I_m,\\ i\in\{0,1\}}} {KZ} (g^i_{m, \lambda})^{-1}. 
\end{eqnarray*}
Thus,  a general element in $\mathrm{ind}_{KZ}^G V$, for a $KZ$-module $V$, is a finite sum of elementary functions of the form $[g,v]$, 
with $g=g_{m,\lambda}^0$ or $\,g_{m,\lambda}^1$, for some $\lambda\in I_m$ and $v\in V$.
For a $\Z_p$-algebra $R$, let $v = \sum_{i=0}^r c_i X^{r-i} Y^i \in V = \mathrm{Sym}^r R^2\otimes D^s$.  
Expanding the formula \eqref{T} for the Hecke operator $T$ one may write
$T  = T^+ + T^-,$ 
where 
\begin{eqnarray*}
  T^+([g^0_{n,\mu},v]) & = & \sum_{\lambda \in I_1} \left[ g^0_{n+1, \mu +p^n\lambda},    
         \sum_{j=0}^r \left( p^j \sum_{i=j}^r c_i {i \choose j}(-\lambda)^{i-j} \right) X^{r-j} Y^j \right], \\
  T^-([g^0_{n,\mu},v]) & = & \left[ g^0_{n-1, [\mu]_{n-1}},    
         \sum_{j=0}^r \left( \sum_{i=j}^r p^{r-i} c_i {i \choose j} 
         \left( \frac{\mu - [\mu]_{n-1}}{p^{n-1}} \right)^{i-j} \right) X^{r-j} Y^j \right] \quad (n > 0), \\
  T^-([g^0_{n,\mu},v]) & = &  [ \alpha,  \sum_{j=0}^r  p^{r-j}  c_j  X^{r-j} Y^j ] \quad  (n=0). 
\end{eqnarray*}
The formulas for $T^+$ and $T^-$ will be used to calculate the $(T-a_p)$-image of the functions $f \in \mathrm{ind}_{KZ}^G \Sym^r\bar\Q_p^2$. 
Though $T$ is a $G$-linear operator, note that $T^+$ and $T^-$ are not $G$-linear. 
Formulas similar to those above describe how $T$ acts on functions of the form $[g_{n,\mu}^1,v]$ but we will not use these functions 
in this article.

An integral function, i.e., an element of $\mathrm{ind}_{KZ}^G\mathrm{Sym}^r\bar\Z_p^2$ will be said to ``die mod $p$'', if it maps to zero in 
$\mathrm{ind}_{KZ}^G\mathrm{Sym}^r\bar\F_p^2$ under the standard reduction map.

\begin{theorem} 
\label{elimination1}
 Let $p\geq 5$, $r>2p$, $r\equiv a\mod (p-1)$ with $3\leq a\leq p-1$. If $r\not\equiv a\mod p$, 
 then there exists a surjection $$\mathrm{ind}_{KZ}^{G}(V_{p-a+1}\otimes
 D^{a-1})\twoheadrightarrow \bar{\Theta}_{k,a_p}.$$
 As a consequence, $\bar{V}_{k,a_p}\cong \mathrm{ind}(\omega_2^{a+p})\:$, if $a>3$. 
 For $a=3$, we have the additional possibility that $\bar{V}_{k,a_p}$ is reducible and its restriction to $I_p$ is
 $\omega^2\oplus\omega^2$.
\end{theorem}

\begin{proof} 
 By Proposition \ref{prop5} (i), we have $Q\cong J_1\oplus J_2$ as a $\Gamma$-module, where 
 $J_1=V_{p-a+1}\otimes D^{a-1}$ and
 $J_2=V_{p-a-1}\otimes D^a$. Let $F_1$ be the image of $\mathrm{ind}_{KZ}^G J_1$ in 
 $\bar{\Theta}_{k,a_p}$ 
 under the surjection $\mathrm{ind}_{KZ}^G Q\twoheadrightarrow \bar{\Theta}_{k,a_p}$, and let  
 $F_2$ denote
 the quotient $\bar{\Theta}_{k,a_p}/F_1$. 
 Then we have the  commutative diagram: 
 \begin{equation*}
    \xymatrix{0 \ar[r] & \mathrm{ind}_{KZ}^G J_1\ar@{>>}[d]\ar[r] & \mathrm{ind}_{KZ}^G Q \ar@{>>}[d] \ar[r] &
       \mathrm{ind}_{KZ}^G J_2 \ar@{>>}[d]\ar[r] & 0 \\
               0 \ar[r] & F_{1}\ar[r]  & \bar{\Theta}_{k,a_p}\ar[r] & F_2 \ar[r] & 0 .}
 \end{equation*}
 We will construct a function in $X_{k,a_p}=\ker (\mathrm{ind}_{KZ}^{G}V_r\twoheadrightarrow 
 \bar{\Theta}_{k,a_p})$, which
 maps to a function $[g,v]\in \mathrm{ind}_{KZ}^{G}J_2$ under the surjection  in the top row above, for some $g\in G$ and $0\neq v\in J_2$. 
 The $G$-span of $[g,v]$ is the whole 
 of $\mathrm{ind}_{KZ}^{G}J_2$, so this implies that $F_2=0$ and $\bar{\Theta}_{k,a_p}\cong F_1$ is a quotient of $\mathrm{ind}_{KZ}^{G}J_1$. 
 Once this is proved, the final conclusion will follow by applying \cite[Prop. 3.3]{[BG09]}. 
 
 Consider the function $f\in \mathrm{ind}_{KZ}^{G}\mathrm{Sym}^r\bar{\mathbb{Q}}_p^2$ defined by $f:=f_2+f_1+f_0$ with
 \begin{eqnarray}\label{f_2}f_2 &=&\underset{\lambda\in\mathbb{F}_p}\sum \left[\:g_{2,p[\lambda]}^0, \frac{1}{p}\cdot\left(Y^r-X^{p-1}Y^{r-p+1}\right)\:\right],\\
 \label{f_1}f_1 &=& \left[\:g_{1,0}^0,\:\:\: \frac{(p-1)}{pa_p}\cdot\sum_{\substack{0<j<r, \\j\equiv a\mod (p-1)}}\:\alpha_j\cdot\, X^{r-j}Y^j\:\right],\\
 \label{f_0}f_0 &=&\begin{cases}0, &\text{if }  a<p-1,\\
 \left[\mathrm{Id}, \frac{(1-p)}{p}\cdot \left(X^r-X^{r-p+1}Y^{p-1}\right)\right], &\text{if } a=p-1,
                \end{cases}
 \end{eqnarray}
 where the $\alpha_j\in\mathbb{Z}$ are integers satisfying the four properties stated in Lemma \ref{comb2}.
 
 Applying the explicit formulas for $T^+$ and $T^-$, using Lemma \ref{comb2} and the facts  $r>2p$, $3\leq p-1$ and $1<v(a_p)<2$, 
 we get that the functions $T^-f_0,\, T^-f_1,\,T^+f_1,\,a_pf_0,\,a_pf_2$ are all integral 
 and die mod $p$. We compute that 
 $T^+f_2\in\mathrm{ind}_{KZ}^{G}\langle X^{r-1}Y\rangle_{\bar{\mathbb{Z}}_p}+p\cdot\mathrm{ind}_{KZ}^{G}\mathrm{Sym}^r\bar{\mathbb{Z}}_p^2$, 
 hence it maps to $0\in\mathrm{ind}_{KZ}^G Q$. 
 Next we use the formulas for $T^-$, $T^+$ and the identity \eqref{identity}  to compute that
 \begin{equation}
  \label{magic} 
    T^-f_2-a_pf_1+T^+f_0=\left[g_{1,0}^0,\:\: \underset{\substack{0<j<r,\\j\equiv a\mod (p-1)}}
    \sum\dfrac{(p-1)}{p}\left(\binom{r}{j}-\alpha_j\right)\cdot X^{r-j}Y^j +Y^r\:\right].
 \end{equation}
 This is also integral as we have $\alpha_j\equiv\binom{r}{j}\mod p$, for each $j$, by Lemma \ref{comb2}.

 All the  information above together implies that $(T-a_p)f\in\mathrm{ind}_{KZ}^{G}\mathrm{Sym}^r\bar{\mathbb{Z}}_p^2$, so its reduction lies in $X_{k,a_p}$. 
 Moreover, the reduction $\overline{(T-a_p)f}$ maps to  $[g_{1,0}^0,,\:\mathrm{Pr}_2(F)]\in\mathrm{ind}_{KZ}^G J_2$, where 
 $$F(X,Y)=\underset{\substack{0<j<r,\\j\equiv a\mod (p-1)}}\sum\dfrac{(p-1)}{p}\left(\binom{r}{j}-\alpha_j\right)\cdot X^{r-j}Y^j +Y^r,$$
 and $\mathrm{Pr}_2:Q\twoheadrightarrow J_2$ is the projection map. 
 It is clear from Diagram \eqref{bigdiadram} that $\mathrm{Pr}_2$ is induced by the map $V_r\twoheadrightarrow\dfrac{V_r}{V_r^*}\twoheadrightarrow J_2$. 
 Noting that the monomial $Y^r$ maps to $0$  in $Q$, we get  that
 $$\mathrm{Pr_2}(F) = \mathrm{Pr}_2  \left( \underset{j} \sum \frac{1}{p} \left(\binom{r}{j}-\alpha_j \right) \cdot X^{r-j}Y^j \right) 
 =\underset{j}\sum\dfrac{1}{p}\left(\binom{r}{j}-\alpha_j\right)\cdot\mathrm{Pr}_2( X^{r-a}Y^a),$$
 since  all the mixed monomials  in $F(X,Y)$ are congruent to  scalar multiples of $X^{r-a}Y^a$ modulo $V_r^*$.
 The composition $V_r/V_r^*\overset{\sim}\longrightarrow V_{a+p-1}/V_{a+p-1}^*\twoheadrightarrow V_{p-a-1}\otimes D^a=J_2$, where the first isomorphism is $\psi^{-1}$ of \cite[(4.2)]{[G78]}
 and the next surjection is induced from \cite[Lem. 5.3]{[Br03]}, sends 
 $$X^{r-a}Y^a\mod V_r^*\mapsto X^{p-1}Y^a\mod V_{a+p-1}^*\mapsto X^{p-1-a}\neq 0.$$
 Therefore $\mathrm{Pr}_2(F)=c\cdot \mathrm{Pr}_2(X^{r-a}Y^a)=c\cdot X^{p-1-a}$, where $c$ is the mod $p$ reduction of the sum  
 $\underset{j}\sum\dfrac{1}{p}\cdot\left(\binom{r}{j}-\alpha_j\right)\in \Z$.
 By  property (ii) of the integers $\alpha_j$ stated in Lemma \ref{comb2},   $c$ is the reduction of
 $ \dfrac{1}{p}\cdot\underset{j}\sum\binom{r}{j}\in\Z$, which is congruent to $\dfrac{a-r}{a}\mod p$ by Lemma \ref{comb1}.
  By hypothesis $r\not\equiv a \mod p$, so we get $c\in\bar\F_p^*$. 
 Thus the reduction  $\overline{(T-a_p)f}$ maps to $[g, v]\in\mathrm{ind}_{KZ}^G J_2$, where $g=g_{1,0}^0\in G$ and $v=c\cdot
 X^{p-1-a}$ is a non-zero element in $J_2$, and we are done. 
\end{proof}

The following theorem is complementary to Theorem \ref{thm1} (ii).

\begin{theorem}\label{thm2}
 Let $p\geq 3$, $r>2p$, $r\equiv 2\mod (p-1)$. If $ p\mid r(r-1)$, then $\bar{V}_{k,a_p}\cong \mathrm{ind}(\omega_2^{2+p})$.
\end{theorem}

\begin{proof}
 If $p\mid r (r-1)$, then $Q$ has two JH factors by  Propositions \ref{pdr1} and \ref{pdr2}. We will eliminate the JH factor 
 $J_2=V_{p-3}\otimes D^2$. We take the  function $f=f_0+f_1+f_2$ as in the proof of Theorem \ref{elimination1}, for $a=2$.
 Note that  property (iv) in Lemma \ref{comb2} implies that $\underset{j}\sum \binom{j}{2}\alpha_j\equiv\binom{r}{2}\equiv 0\mod p$, by hypothesis. 
 Now by the same argument as in
 Theorem \ref{elimination1}, we eliminate the JH factor $J_2$, except for the following subtlety: to show $T^-f_1\equiv 0\mod p$, 
 we used the bound $3\leq p-1$ in Theorem \ref{elimination1}. This cannot be used 
 in the present case
 as $a=2$ and so $p=3$ is allowed. But $p\mid r(r-1)$ implies that $\alpha_{r-p+1}\equiv \binom{r}{r-p+1}\equiv 0\mod p$ by Lucas' theorem. This ensures that $T^-f_1$ dies mod $p$ even in this case.
 
 Therefore if $p\mid r-1$, then we have a surjection $\mathrm{ind}_{KZ}^G (V_{p-1}\otimes D)\twoheadrightarrow \bar{\Theta}_{k,a_p}$ by Proposition \ref{pdr1}, 
 and if $p\mid r$, then we have $\mathrm{ind}_{KZ}^G (V_0\otimes D)\twoheadrightarrow \bar{\Theta}_{k,a_p}$ by Proposition \ref{pdr2}. 
 Now we use \cite[Prop. 3.3]{[BG09]} to draw the final conclusion.
\end{proof}

Next we treat some cases where $Q$ has three JH factors and those coming from $V_r^*/V_r^{**}$ are to be eliminated. We begin by stating the following easy lemma. 
Note that we already have a complete solution to the problem in
the case $b=2$ by Theorem \ref{thm1} (ii) and Theorem \ref{thm2}. Therefore we assume $ b\geq 3$ from now on. 

\begin{lemma}
\label{easylemma}
 Let $p\geq 3$, $r>2p$, $ r\equiv b\mod (p-1)$ with $3\leq b\leq p$. Then we have the non-split short exact sequence of $\,\Gamma$-modules 
 $$0\rightarrow J_0:=V_{b-2}\otimes D\rightarrow V_r^*/V_r^{**}\rightarrow J_1:=V_{p-b+1}\otimes D^{b-1}\rightarrow 0,\text{ where }$$
 \begin{enumerate}\item[(i)] The monomials $Y^{b-2},X^{b-2}\in J_0$ map to  $\theta Y^{r-p-1}$ and $\theta X^{r-p-1}$ respectively in $V_r^*/V_r^{**}$. 
  \item[(ii)] The polynomials $\theta Y^{r-p-1},\,\theta X^{r-p-1}\in V_r^*/V_r^{**}$ map to $0\in J_1$ and  $\,\theta X^{r-p-b+1}Y^{b-2}$ maps to $X^{p-b+1}\in J_1$.
 \end{enumerate}
\end{lemma}

\begin{proof}
 The exact sequence is given by Proposition \ref{es2}. By \cite[(4.1), (4.2)]{[G78]}, we have an isomorphism 
 $V_r^*/V_r^{**}\xrightarrow{\psi^{-1}\otimes \mathrm{id}} (V_{p+b-3}/V_{p+b-3}^*)\otimes D$. 
 Then one computes the images of the polynomials mentioned above under the $\Gamma$-maps $V_{b-2}\otimes D\hookrightarrow (V_{p+b-3}/V_{p+b-3}^*)\otimes D$ and 
 $(V_{p+b-3}/V_{p+b-3}^*)\otimes D\twoheadrightarrow V_{p-b+1}\otimes D^{b-1}$ respectively,  using the explicit formulas from \cite[Lem. 5.3]{[Br03]} .
\end{proof}

The next two theorems are complementary to Theorem \ref{elimination1} above.

\begin{theorem}\label{elimination3}
  Let $p\geq 5$, $r>2p$ and $r\equiv b\mod (p-1)$ with $4\leq b\leq p-1$. If $r\equiv b\mod p$, then $\bar{V}_{k,a_p}\cong \mathrm{ind}(\omega_2^{b+1})$.  
\end{theorem}

\begin{proof} 
 By Proposition \ref{prop5} (ii),  $Q$ contains $V_r^*/V_r^{**}$ as a submodule, which is an extension of $J_1$ by $J_0$, 
 with the notation of Lemma \ref{easylemma}. Let $F_{0,1}$ denote the image of $\mathrm{ind}_{KZ}^G(V_r^*/V_r^{**})$ inside $\bar{\Theta}_{k,a_p}$, and let 
 $F_2:=\bar{\Theta}_{k,a_p}/F_{0,1}$.  Then we have the following commutative diagram of $G$-maps
 \begin{equation*}
    \xymatrix{0 \ar[r] & \mathrm{ind}_{KZ}^G(V_r^*/V_r^{**})\ar@{>>}[d]\ar[r] & \mathrm{ind}_{KZ}^G Q \ar@{>>}[d] \ar[r] & \mathrm{ind}_{KZ}^G J_2 \ar@{>>}[d]\ar[r] & 0 \\
               0 \ar[r] & F_{0,1}\ar[r]  & \bar{\Theta}_{k,a_p}\ar[r] & F_2 \ar[r] & 0 ,}
 \end{equation*} where $J_2=V_{p-b-1}\otimes D^{b}$.
 We will show that $F_{0,1}=0$, under the hypothesis $r \equiv b \mod p$.

 Consider $f=f_0+f_1\in\mathrm{ind}_{KZ}^G\mathrm{Sym}^r\bar{\mathbb{Q}}_p^2$, given by
 \begin{eqnarray*}
  f_1 &=&\underset{\lambda\in\mathbb{F}_p^*}\sum\left[\,g_{1,[\lambda]}^0,\:\frac{p}{a_p}[\lambda]^{p-2}\cdot\left(Y^r-X^{r-b}Y^b\right)\right]
               +\left[\,g_{1,0}^0,\:\frac{r(1-p)}{a_p}\cdot\left(XY^{r-1}-X^{r-b+1}Y^{b-1}\right)\right],\\
  f_0 &=&\left[\mathrm{Id},\:\sum_{\substack{0<j<r-1, \\j\equiv b-1\mod (p-1)}}\:\frac{p(p-1)}{a_p^2}\cdot\beta_j\, X^{r-j}Y^j\:\right],
 \end{eqnarray*}
 where the $\beta_j$ are the integers from Lemma \ref{comb3}.

 Using $b>2$ and the fact that $p\mid r-b$, we check that $T^+f_1\equiv 0\mod p$. Similarly, $T^-f_0\equiv 0\mod p$, since $v(a_p^2/p)<3\leq p$. 
 We use the fact  $b>3$, together with
 the properties satisfied by the integers $\beta_j$ in Lemma \ref{comb3} to conclude that $T^+f_0\equiv 0\mod p$ as well. Next we compute that 
 $$T^-f_1-a_pf_0\equiv \left[\mathrm{Id}, \underset{\substack{0<j<r-1\\j\equiv b-1\mod(p-1)}}\sum\frac{(p-1)p}{a_p}\left(\binom{r}{j}-\beta_j\right)\cdot X^{r-j}Y^j\right],$$
 which again dies mod $p$, since $\beta_j\equiv \binom{r}{j}\mod p$ for each $j$, by Lemma \ref{comb3}. 
 Finally we get  $(T-a_p)f$ is integral and $(T-a_p)f\equiv -a_pf_1\mod p$. As $r\equiv b\mod p$, we have
 \begin{eqnarray*} 
   (T-a_p)f &\equiv & -\left[g_{1,0}^0,\: r(1-p)\cdot(XY^{r-1}-X^{r-b+1}Y^{b-1})\right]\\
            &\equiv &-\left[g_{1,0}^0,\:-b\cdot\theta\,\left(\dfrac{r-p-b+1}{p-1}\cdot X^{r-p-b+1}Y^{b-2}+Y^{r-p-1}\right)\mod V_r^{**}\right]\\
            &\equiv &\left[g_{1,0}^0,\:-b\cdot\theta\,\left(X^{r-p-b+1}Y^{b-2}-Y^{r-p-1}\right)\right]\mod p.
 \end{eqnarray*}
 Let $v$ be the image of $-b\cdot \theta\left(X^{r-p-b+1}Y^{b-2}-Y^{r-p-1}\right)$ in $V_r^*/V_r^{**}$. Then the reduction $\overline{(T-a_p)f}$ maps to
 $\left[g_{1,0}^0,\: v\right]\in \mathrm{ind}_{KZ}^G(V_r^*/V_r^{**})\subseteq 
 \mathrm{ind}_{KZ}^G Q$.
 By Lemma \ref{easylemma}, $v$ maps to the non-zero element $-b\cdot X^{p-b+1}\in J_1=V_{p-b+1}\otimes D^{b-1}$. 
 As the short exact sequence \eqref{ses**3} is non-split, $v$ generates the whole module $V_r^*/V_r^{**}$ over
 $\Gamma$, so the element $\left[g_{1,0}^0,\: v\right]$ generates $\mathrm{ind}_{KZ}^G(V_r^*/V_r^{**})$ over $G$. 
 Thus $F_{0,1}=0$ and hence $\bar{\Theta}_{k,a_p}\cong F_2$ is a quotient of 
 $\mathrm{ind}_{KZ}^G(V_{p-b-1}\otimes D^{b})$. Finally we apply \cite[Prop. 3.3]{[BG09]} to get the structure of $\bar{V}_{k,a_p}$.
\end{proof}

However, for  $b=3\leq p-1$, we do not have a complete solution to the problem of computing $\bar{V}_{k,a_p}$ when $r\equiv b\mod p(p-1)$. 
But we have the following theorem, which gives a generic answer. 
The theorem below is applicable whenever $v(a_p)\neq \frac{3}{2}$. 
It is also applicable if $v(a_p)=\frac{3}{2}$, unless the unit $\frac{a_p^2}{p^3}$ reduces to $1$ in $\bar\F_p$.

\begin{theorem}
  \label{elimination6}
  Let $p\geq 5$, $r>2p$ and $r\equiv 3\mod p(p-1)$. If $v(a_p) = \frac{3}{2}$, then assume that
  $\,v(a_p^2-p^3) = 3$. Then $\bar{V}_{k,a_p}\cong \mathrm{ind}(\omega_2^4)$.
\end{theorem}

\begin{proof} 
 If $v(a_p)\leq 3/2$, then we consider  $f=f_0+f_1$, where $f_0$ are $f_1$ are as in the proof of Theorem \ref{elimination3} with $b=3$. 
 The formula for the Hecke operator shows that 
 $(T-a_p)f$ is still integral. As $b=3$, now $T^+f_0$ does not necessarily die mod $p$.  In fact we have
 $T^+f_0\equiv \left[g_{1,0},\dfrac{p^3(p-1)}{a_p^2}\binom{2}{2}\beta_2\cdot X^{r-2}Y^2\right]\mod p$, which is integral because 
 $v(a_p^2)\leq 3$, 
 and we have
 $$ (T-a_p)f\equiv T^+f_0-a_pf_1\equiv \left[g_{1,0},\:\dfrac{p^3(p-1)}{a_p^2}\,\beta_2\cdot X^{r-2}Y^2-r(1-p)\left(XY^{r-1}-X^{r-2}Y^2\right)\right]\mod p.$$
 Note that $r\equiv 3\mod p$ and also $\beta_2\equiv\binom{r}{2}\equiv\binom{3}{2}=3\mod p$ by Lucas' theorem. As $XY^{r-1}$ vanishes in $Q$, 
 the reduction $\overline{(T-a_p)f}$ maps to the image of  $\left[g_{1,0}^0,\,3\left(1-{p^3}/{a_p^2}\right)\left(X^{r-2}Y^2-XY^{r-1}\right)\right]$ 
 in $\mathrm{ind}_{KZ}^G Q$. The hypothesis implies that $1-p^3/a_p^2$ is a $p$-adic
 unit. So, the module $\mathrm{ind}_{KZ}^G(V_r^*/V_r^{**})$ maps to $0$ in $\bar{\Theta}_{k,a_p}$, as explained in the proof of Theorem \ref{elimination3}. 

 If $v(a_p)>3/2$, then we consider the new function $f^\prime=\dfrac{a_p^2}{p^3}\cdot f=f_0^\prime+f_1^\prime,\,$ given by
 \begin{eqnarray*}
  f_1^\prime &=&\underset{\lambda\in\mathbb{F}_p^*}\sum\left[\,g_{1,[\lambda]}^0,\:\frac{a_p}{p^2}\,[\lambda]^{p-2}\cdot\left(Y^r-X^{r-3}Y^3\right)\right]
   +\left[\,g_{1,0}^0,\:\frac{(1-p)r a_p}{p^3}\cdot\left(XY^{r-1}-X^{r-2}Y^{2}\right)\right],\\
   f_0^\prime &=& \left[\mathrm{Id},\:\sum_{\substack{0<j<r-1, \\j\equiv 2\mod (p-1)}}\:\frac{(p-1)}{p^2}\cdot\beta_j\, X^{r-j}Y^j\:\right],
 \end{eqnarray*}
 where the $\beta_j$ are the integers from Lemma \ref{comb3}. The computations are very similar to the previous case, except that 
 now we have $a_pf_1^\prime\equiv 0\mod p$, since $v(a_p^2/p^3)>0$. Finally we get 
 $(T-a_p)f^\prime\equiv T^+f_0^\prime\equiv\left[g_{1,0}^0,\,(p-1)\beta_2\cdot X^{r-2}Y^2\right]\mod p$. So the
 reduction $\overline{(T-a_p)f^\prime}$ maps to the image of $ \left[g_{1,0}^0,\,3(XY^{r-1}-X^{r-2}Y^2)\right]$ in $\mathrm{ind}_{KZ}^G Q$. The rest of the proof follows 
 as in the previous case. 
\end{proof}
\begin{remark}
        Note that when $b = 3$ and $p \> | \> r-b$,  
        the hypothesis \eqref{hyp} in Theorem \ref{theorem-intro} is equivalent to the condition $\,v(a_p^2-p^3) = 3$ above.
        If $v(a_p^2-p^3) > 3$, so necessarily $v(a_p) = \frac{3}{2}$, we can only show that 
        $\bar V_{k,a_p}$ is  either $\mathrm{ind}(\omega_2^{4})$ or
        $\mathrm{ind}(\omega_2^{3+p})$, 
        or it is reducible
        of the form $\omega^2\oplus \omega^2$ or $\omega^3\oplus \omega$ on $I_p$. 
\end{remark}

The following theorem is complementary to Theorem~\ref{thm1} (i).  
It treats the case $p\mid r\equiv 1\mod p-1$, where $Q$ has two JH factors. 
With the notation of Lemma \ref{comb3}, $a=1$ is equivalent to $b=p$, and so the condition $p\mid r$ can also be stated as 
$r\equiv b\mod p$.  

\begin{theorem}
 \label{elimination4}
 For $p\geq 3$, let $p<r\equiv 1\mod (p-1)$  and suppose $p \mid r$. If $p=3$ and $v(a_p)= \frac{3}{2}$, then further assume that $v(a_p^2-p^3)=3$.
 \begin{enumerate}\item[(i)] If $p^2\nmid (p-r)$, then there is a surjection $\mathrm{ind}_{KZ}^G V_{1}\twoheadrightarrow\bar{\Theta}_{k,a_p}$. As a consequence, 
 $\bar{V}_{k,a_p}\cong \mathrm{ind}(\omega_2^2)$ unless $p=3$, in which case $\bar{V}_{k,a_3}$ may be
 reducible and trivial on $I_3$.
  \item[(ii)] If  $p^2\mid (p-r)$, then there is a surjection $\mathrm{ind}_{KZ}^G(V_{p-2}\otimes D)\twoheadrightarrow\bar{\Theta}_{k,a_p}$. As a consequence, 
  either $\bar{V}_{k,a_p}\cong \mathrm{ind}(\omega_2^{2})$ or $\bar{V}_{k,a_p}$ is reducible with the shape $\omega\oplus\omega$ on $I_p$.
 \end{enumerate}
\end{theorem}

\begin{proof}
 By Proposition \ref{prop4} (iii), we have $Q\cong V_r^*/V_r^{**}$ is an extension of $J_1=V_1$ by $J_0=V_{p-2}\otimes D$. 
 Let $F_0\subseteq\bar{\Theta}_{k,a_p}$ be the image of $\mathrm{ind}_{KZ}^G J_0$ under the 
 map $\mathrm{ind}_{KZ}^G Q\twoheadrightarrow\bar{\Theta}_{k,a_p}$. Then  $F_1:=\bar{\Theta}_{k,a_p}/F_0$ is a quotient of $\mathrm{ind}_{KZ}^G J_1$ 
 and we have the following commutative diagram:
 \begin{eqnarray*}
    \xymatrix{0 \ar[r] & \mathrm{ind}_{KZ}^G J_0\ar@{>>}[d]\ar[r] & \mathrm{ind}_{KZ}^G Q \ar@{>>}[d] \ar[r] & \mathrm{ind}_{KZ}^G J_1 \ar@{>>}[d]\ar[r] & 0 \\
               0 \ar[r] & F_0\ar[r]  & \bar{\Theta}_{k,a_p}\ar[r] & F_1 \ar[r] & 0 .}
 \end{eqnarray*}
 We will show that $F_0=0$ if $p^2\nmid (p-r)$, and $F_1=0$ if $p^2\mid (p-r)$.
 \begin{enumerate}
 \item[(i)] Consider $f=f_2+f_1+f_0\in\mathrm{ind}_{KZ}^G\mathrm{Sym}^r\bar{\mathbb{Q}}_p^2$, given by
 \begin{eqnarray*}
 f_2 &=&\underset{\lambda\in\mathbb{F}_p}\sum \left[\:g_{2,p[\lambda]}^0, \frac{[\lambda]^{p-2}}{p}\cdot(Y^r-X^{r-p}Y^{p})\:\right],\\
 f_1 &=& \left[\:g_{1,0}^0,\:\:\: \frac{p-1}{pa_p}\cdot\sum_{\substack{0<j<r-1, \\j\equiv 0\mod (p-1)}}\:\beta_j\cdot\, X^{r-j}Y^j\:\right],\\
 f_0 &=& \left[\mathrm{Id}, \frac{1-p}{p}\cdot (X^r-X^{p}Y^{r-p})\right],
 \end{eqnarray*} 
 where the $\beta_j$ are the integers from Lemma \ref{comb3}. 
 Using the formula for the Hecke operator, one checks that $T^+f_2,T^-f_1,-a_pf_2, T^-f_0, -a_pf_0$ are all integral and die mod $p$.
 Also $T^+f_1\equiv 0\mod p$ by Lemma \ref{comb3}, i.e., by the properties satisfied by the integers $\beta_j$.
 Moreover, 
 $$\label{T-a_p}T^-f_2+T^+f_0-a_pf_1\equiv\left[\:g_{1,0}^0,\:\: \dfrac{(p-1)}{p}\cdot\left(\underset{\substack{0<j<r-1\\j\equiv 0\mod (p-1)}}\sum \left(\binom{r}{j}-\beta_j\right)\cdot X^{r-j}Y^j+r XY^{r-1}\right)\:\right]\mod p.$$
 Note that the function above is integral because each $\beta_j\equiv \binom{r}{j}\mod p$ by Lemma \ref{comb3}, and  $p\mid r$ by hypothesis. 
 Now modifying the polynomial above by a suitable $XY^{r-1}$-term, we can see that $\overline{(T-a_p)f}$ 
 has the same image as $\left[g_{1,0}^0,\:\: (p-1)\left(F(X,Y)+\dfrac{(p-r)\theta Y^{r-p-1}}{p}\right)\right]$ in $\mathrm{ind}_{KZ}^G Q$, where 
 $$F(X,Y)=\underset{\substack{0<j<r-1\\j\equiv 0\mod (p-1)}}\sum \frac{1}{p}\left(\binom{r}{j}-\beta_j\right)\cdot X^{r-j}Y^j-\frac{(p-r)}{p}\cdot X^pY^{r-p}.$$
 Now we apply Lemma \ref{elelemma} to conclude that $F(X,Y)\in V_r^{**}$, using Lemma \ref{comb3},  \ref{comb4} and 
 \ref{comb3a}. 
 Hence $\overline{(T-a_p)f}$ maps to the image of 
 $\dfrac{(r-p)}{p}\cdot\left[g_{1,0}^0,\:\theta Y^{r-p-1}\right]$ in $\mathrm{ind}_{KZ}^G Q$. By hypothesis, $p^2\nmid (p-r)$. Thus 
 $c=\overline{(r-p)/p}$ is a non-zero element in $\bar\F_p$.
 By Lemma \ref{easylemma} (i), the element  $c\cdot[g_{1,0}^0,Y^{p-2}]\in \mathrm{ind}_{KZ}^G J_0$ maps to 
 $0 \in F_0\subseteq \bar{\Theta}_{k,a_p}$.
 Since $c\cdot[g_{1,0}^0,Y^{p-2}]$ generates all of $\mathrm{ind}_{KZ}^G J_0$ as a $G$-module, we have $F_0=0$ and 
 $\bar{\Theta}_{k,a_p}\cong F_1$ is a quotient of $\mathrm{ind}_{KZ}^G J_1$. Now the result follows from  \cite[Prop. 3.3]{[BG09]}.

 \item[(ii)] If $p^2\mid (p-r)$, we consider the function $f=f_2+f_1+f_0\in\mathrm{ind}_{KZ}^G\mathrm{Sym}\,\bar{\mathbb{Q}}_p^2$, given by
 \begin{eqnarray*}
 f_2 &=& \underset{\lambda\in\mathbb{F}_p^*}\sum \left[\:g_{2,p[\lambda]}^0, \frac{[\lambda]^{p-2}}{a_p}\cdot(Y^r-X^{r-p}Y^{p})\:\right]+\left[g_{2,0}^0,\:\frac{(1-p)r}{pa_p}\cdot(XY^{r-1}-X^{r-p+1}Y^{p-1})\right],\\
 f_1 &=& \left[\:g_{1,0}^0,\:\:\: \frac{p-1}{a_p^2}\cdot\sum_{\substack{0<j<r-1, \\j\equiv 0\mod (p-1)}}\:\gamma_j\cdot\, X^{r-j}Y^j\:\right],\\
 f_0 &=& \left[\mathrm{Id},\: \frac{1-p}{a_p}\cdot (X^r-X^{p}Y^{r-p})\right],
 \end{eqnarray*} 
 where the $\gamma_j$ are  integers from Lemma \ref{comb5} (ii). Using Lemma \ref{comb5} we check that if either $p\geq 5$ or if $p=3$ and $v(a_p)<3/2$, then
 $T^+f_2, T^+f_0+T^-f_2-a_pf_1, T^-f_1, T^+f_1, T^-f_0$ are all integral 
 and die mod $p$. That leaves us with $(T-a_p)f \equiv  -a_pf_0 - a_pf_2 \equiv$
 \begin{eqnarray*}
 \underset{\lambda\in\mathbb{F}_p^*}\sum \left[\:g_{2,p[\lambda]}^0, -[\lambda]^{p-2}\cdot(Y^r-X^{r-p}Y^{p})\:\right] 
 + \left[g_{2,0}^0,\:\dfrac{(p-1)r}{p}(XY^{r-1}-X^{r-p+1}Y^{p-1})\right] \\
  \qquad\qquad  +\left[\mathrm{Id},\: (p-1)\cdot (X^r-X^{p}Y^{r-p})\right]\mod p.
 \end{eqnarray*}
 Therefore the image of the reduction $\overline{(T-a_p)f}$ in $\mathrm{ind}_{KZ}^G Q$ is the same as that of
 \begin{eqnarray*}
  \underset{\lambda\in\mathbb{F}_p^*}\sum [\:g_{2,p[\lambda]}^0, -[\lambda]^{p-2}\cdot(X^{r-1}Y-X^{r-p}Y^{p})\:]
 &+& [\:g_{2,0}^0,\:\dfrac{(p-1)r}{p}(XY^{r-1}-X^{r-p+1}Y^{p-1})]\\
 & +& [\mathrm{Id},\: (p-1)\cdot (XY^{r-1}-X^{p}Y^{r-p})].\end{eqnarray*}
 As we have $r/p\equiv 1\mod p$ by hypothesis, the function above is congruent to 
 $$ \underset{\lambda\in\mathbb{F}_p^*}\sum [\:g_{2,p[\lambda]}^0, -[\lambda]^{p-2}\theta X^{r-p-1}\:]
 +[g_{2,0}^0,\,X^{r-p+1}Y^{p-1}-XY^{r-1}]
 +[\mathrm{Id},\, \theta Y^{r-p-1}]\mod p.$$
 Since  $X^{r-p+1}Y^{p-1}-XY^{r-1}=\theta\cdot(X^{r-2p+1}Y^{p-2}+\cdots +Y^{r-p-1})$,  we have
 $$X^{r-p+1}Y^{p-1}-XY^{r-1}\equiv\theta\cdot\left(\dfrac{(r-2p+1)}{p-1}\cdot X^{r-2p+1}Y^{p-2}+Y^{r-p-1}\right) \mod V_r^{**}.$$
 Applying Lemma \ref{easylemma} (ii) we get that $\overline{(T-a_p)f}$ maps to $\left[g_{2,0}^0,\:-X\right] \in\mathrm{ind}_{KZ}^G J_1$ under the map
 $\mathrm{ind}_{KZ}^G Q\twoheadrightarrow \mathrm{ind}_{KZ}^G \, J_1$. 
 As $[g_{2,0}^0,-X]$ generates all of $\mathrm{ind}_{KZ}^{G}J_1$ as a $G$-module, we get $F_1=0$ and so $\bar{\Theta}_{k,a_p}\cong F_0$ 
 is a quotient of $\mathrm{ind}_{KZ}^G J_0$. 

 If $p=3$ and $v(a_p)\geq 3/2$, then note that $T^-f_1$ and $T^+f_1$ do not die mod $p$ any more, and $(T-a_p)f$ is not necessarily integral. 
 In this case we consider the modified new function $f^\prime:=(a_p^2/p^3)\cdot f$, with $f$ as above. Then one checks $(T-a_p)f^\prime$ is integral and 
 maps to $c\cdot [g_{2,0}^0, X]\in\mathrm{ind}_{KZ}^G J_1$, where $c=\overline{1-a_p^2/p^3}$. By the extra hypothesis in the case $v(a_p)=3/2$,  
 $c$ is always a non-zero element in $\bar\F_p$, and thus the JH factor $J_1$ is killed again. 
 
 Now the result
 follows from \cite[Prop. 3.3]{[BG09]}.
\end{enumerate}
\end{proof}

\begin{remark}
 In the next section we will show that in part (i) above the reducible case does not occur even if $p = 3$, and that in part (ii) $\bar{V}_{k,a_p}$ is always reducible.
\end{remark}

\section{Separating out reducible and irreducible cases}
\label{sep}

If $\bar{\Theta}_{k,a_p}$ is a quotient of $\mathrm{ind}_{KZ}^G (V_{p-2}\otimes D^n)$, then  \cite[Prop. 3.3]{[BG09]} fails to 
determine $\bar{V}_{k,a_p}$. In this case, either $\bar{V}_{k,a_p}\cong \mathrm{ind}(\omega_2^{p-1+n(p+1)})$ is irreducible or
it is reducible with $\bar{V}_{k,a_p}|_{I_p}\cong\omega^n\oplus\omega^n$. We have faced this problem in the following cases, 
cf. Theorems \ref{thm1} (i),  \ref{elimination1} and  \ref{elimination4}.
\begin{enumerate}
 \item If $b=3$ and $p^{1+v(b)}\nmid (r-b)$, then we have $\mathrm{ind}_{KZ}^G\left(V_{p-2}\otimes D^2\right)\twoheadrightarrow \bar{\Theta}_{k,a_p}$, 
 hence $$\bar{V}_{k,a_p}|_{I_p}\cong \begin{cases}
                                   \mathrm{ind}(\omega_2^{p+3}),  \text{ or }\\
                                     \omega^2\oplus\omega^2.
                                      \end{cases}$$
 \item If  $b=p$ and $p^2\mid (r-b)$,  then  under the extra hypothesis `$v(a_p)=3/2 \Rightarrow v(a_p^2-p^3)=3$ when $p=3$', 
 we have $\mathrm{ind}_{KZ}^G\left(V_{p-2}\otimes D\right)\twoheadrightarrow \bar{\Theta}_{k,a_p}$, 
 hence $$\bar{V}_{k,a_p}|_{I_p}\cong \begin{cases}
                                   \mathrm{ind}(\omega_2^{2}),  \text{ or }\\
                                   \omega\oplus\omega.\end{cases}$$
 \end{enumerate}

In this section we mostly 
separate out the reducible and irreducible possibilities above. 
We will show that $\bar{V}_{k,a_p}$ is \textquoteleft almost always\textquoteright $\,$ 
irreducible in the first case whereas it is always
reducible in the second case above. 

In the first case, we work under the mild hypothesis \eqref{hyp} in 
Theorem \ref{theorem-intro}. Note that \eqref{hyp} holds trivially 
if $p \mid {r-1 \choose 2}$. In particular, it holds for the smallest new weight 
treated  in this paper, namely, $k = 2p+3$. 

\begin{theorem}
 \label{reducible1}
 Let $p\geq 3$, $r>2p$, $r\equiv 3\mod (p-1)$ and  $p^{1+v(3)}\nmid(r-3)$. If $v(a_p)= \frac{3}{2}$,  then further assume that  
 $v\left(a_p^2-\binom{r-1}{2}(r-2)p^3\right)= 3$. Then 
 $\bar{V}_{k,a_p}\cong\mathrm{ind}(\omega_2^{p+3})$ is irreducible. 
\end{theorem}

\begin{proof} Assuming  the hypothesis, we will show that the $G$-map  
 $\mathrm{ind}_{KZ}^G J_1 \twoheadrightarrow \bar{\Theta}_{k,a_p}$ given by Propositions \ref{prop1}, \ref{prop2} and Theorem
 \ref{elimination4} (i) for $p=3$, and by 
 Theorem \ref{elimination1} for $p\geq 5$, factors through the cokernel of 
 the Hecke operator $T$ acting on $\mathrm{ind}_{KZ}^G J_1 $, 
 where $J_1=V_{p-2}\otimes D^2$.

 If $v(a_p^2)\leq v\left(\binom{r-1}{2}\right)+3$, 
 we consider the function $f=f_2+f_1+f_0$, where
 \begin{eqnarray*}
  f_2 &=& \underset{\lambda\in\mathbb{F}_p}\sum\left[g_{2,p[\lambda]}^0,\:\frac{1}{a_p}\cdot\theta(X^{r-p-2}Y-Y^{r-p-1})\right],\\
  f_1 &=& \left[g_{1,0}^0, \:\underset{\substack{0<j<r-1\\j\equiv 2\mod (p-1)}}\sum\frac{(p-1)p}{a_p^2}\cdot\alpha_j X^{r-j}Y^j\right],\\
  f_0 &=& \begin{cases}
              0, & \text{ if } p\geq 5\\
              \left[\mathrm{Id},\:\frac{(1-p)p}{a_p}\cdot(X^r-X^{r-p+1}Y^{p-1})\right], & \text{ if } p=3.
           \end{cases}
 \end{eqnarray*}
 where the $\alpha_j$ are integers from Lemma \ref{comb2} for $r-1\equiv 2\mod (p-1)$. 

 It is easy to see using the formula for the Hecke operator that $T^+f_2$, $T^-f_1$ are integral and die mod $p$.
 Moreover $T^-f_2-a_pf_1+T^+f_0\equiv \left[g_{1,0}^0,\, \underset{\substack{0<j<r-1\\j\equiv 2\mod (p-1)}}\sum\dfrac{(p-1)p}{a_p}\cdot\left(\binom{r-1}{j}-\alpha_j\right)X^{r-j}Y^j\right]$ also dies mod $p$, since $v(a_p)<2$
 and each $\alpha_j\equiv \binom{r-1}{j}\mod p$, by Lemma \ref{comb2}. Also, $T^-f_0$ and $a_pf_0$ dies mod $p$, which are relevant only when $p=3$.  
 We use the four properties of the $\alpha_j$ in Lemma \ref{comb2}, and the fact that $v(a_p^2)\leq v\left(\binom{r-1}{2}\right)+3$ to conclude that $T^+f_1$ 
 is also integral and that
 $T^+f_1\equiv\underset{\lambda\in\mathbb{F}_p}\sum\left[g_{2,p[\lambda]}^0,\:\dfrac{p^3(p-1)}{a_p^2}\binom{r-1}{2}\cdot X^{r-2}Y^2\right]\mod p$. 
 Thus $(T-a_p)f$ is integral and
$$(T-a_p)f\equiv \underset{\lambda\in\mathbb{F}_p}\sum\left[g_{2,p[\lambda]}^0,\:-\theta(X^{r-p-2}Y-Y^{r-p-1})+\dfrac{p^3(p-1)}{a_p^2}\binom{r-1}{2}\cdot X^{r-2}Y^2\right]\mod p.$$
 Note that the image of $X^{r-2}Y^2$ in $Q$ is the same as that of  
 $$ X^{r-2}Y^2-XY^{r-1}=\theta\cdot(X^{r-p-2}Y+\cdots +Y^{r-p-1})\equiv\theta\cdot\left(\frac{r-p-2}{p-1}\cdot X^{r-p-1}Y+Y^{r-p-1}\right)\mod V_r^{**}.$$
 Hence   $\overline{(T-a_p)f}$ maps to  
 $\underset{\lambda\in\mathbb{F}_p}\sum\left[g_{2,p[\lambda]}^0,\:-X^{p-2}+\frac{p^3}{a_p^2}\binom{r-1}{2}(r-2)\cdot X^{p-2}\right]\in \mathrm{ind}_{KZ}^G (V_{p-2}\otimes D^2)$ by Lemma \ref{easylemma}. 
 This equals $c \cdot T([g_{1,0}^0,\, X^{p-2}])$, with $c=\overline{\frac{p^3}{a_p^2}\binom{r-1}{2}(r-2)-1}$. 
 Using the hypothesis one checks that the constant $c\in\F_p$ is non-zero, 
  hence the map 
 $\mathrm{ind}_{KZ}^G \, J_1 \twoheadrightarrow \bar{\Theta}_{k,a_p}$ factors through $\pi(p-2, 0,\omega^2)$. Therefore the reducible case cannot occur and
 $\bar{V}_{k,a_p}\cong \mathrm{ind}\left(\omega_2^{p+3}\right)$.

 Now let $v(a_p^2)> v\left(\binom{r-1}{2}\right)+3$. As $v(a_p^2)<4$, this forces  $p\nmid\binom{r-1}{2}$ and so  $v(a_p^2)>3$. 
 Note that in this case  $(T-a_p)f$ is not integral for the $f$ above. However, we can use the following  
 modified function 
 $f^\prime=f_2^\prime+f_1^\prime+f_0^\prime$ :
 \begin{eqnarray*}
  f_2^\prime=\frac{a_p^2}{p^3}\cdot f_2 &=& \underset{\lambda\in\mathbb{F}_p}\sum\left[g_{2,p[\lambda]}^0,\:\frac{a_p\cdot\theta(X^{r-p-2}Y-Y^{r-p-1})}{p^3}\right],\\
  f_1^\prime=\frac{a_p^2}{p^3}\cdot f_1 &=& \left[g_{1,0}^0, \:\underset{\substack{0<j<r-1\\j\equiv 2\mod (p-1)}}\sum\frac{(p-1)}{p^2}\cdot\alpha_j X^{r-j}Y^j\right],\\
  f_0^\prime =\frac{a_p^2}{p^3}\cdot f_0 &=& \begin{cases}
                                                 0, & \text{ if } p\geq 5\\
                                                 \left[\mathrm{Id},\:\frac{(1-p)a_p}{p^2}\cdot(X^r-X^{r-p+1}Y^{p-1})\right], & \text{ if } p=3.
                                             \end{cases}
 \end{eqnarray*}
 with the same $\alpha_j$'s as before. 
 Now  $a_pf_2^\prime$, $a_pf_0^\prime$ dies mod $p$, as $v(a_p^2)>3$. Also $T^+f_2^\prime$, $T^-f_1^\prime$, $T^-f_0^\prime$ 
 and $T^-f_2^\prime-a_pf_1^\prime+T^+f_0^\prime$  die mod $p$ as before, 
 and hence 
 $$(T-a_p)f^\prime\equiv T^+f_1^\prime\equiv\underset{\lambda\in\mathbb{F}_p}\sum\left[g_{2,p[\lambda]}^0,\:(p-1)\binom{r-1}{2}X^{r-2}Y^2\right]\mod p.$$ 
 This maps to 
 $\underset{\lambda\in\mathbb{F}_p}\sum\left[g_{2,p[\lambda]}^0,\:\overline{(r-2)\binom{r-1}{2}}\cdot X^{p-2}\right]= \overline{(r-2)\binom{r-1}{2}}\cdot \,T([g_{1,0}^0,\, X^{p-2}]$
 under the map 
 $\mathrm{ind}_{KZ}^G V_r\twoheadrightarrow\mathrm{ind}_{KZ}^G J_1$,
 as shown above. Since $(r-2)\binom{r-1}{2}$ is
 a $p$-adic unit, the map $\mathrm{ind}_{KZ}^G J_1 \twoheadrightarrow\bar\Theta_{k,a_p}$ factors through the image of $T$, 
 and we have $\bar{V}_{k,a_p}\cong \mathrm{ind}\left(\omega_2^{p+3}\right)$.
\end{proof}


Next we consider the  case $b=p$ and $p^2\mid (p-r)$. Surprisingly, $\bar{V}_{k,a_p}$ is always reducible in this case, at least if $p \geq 5$. 
This is the first time we have obtained a family of examples where  $\bar{V}_{k,a_p}$ is reducible,  under the hypothesis $1<v(a_p)<2$. 
The following theorem describes the action of both inertia and Frobenius elements.

\begin{theorem}
 \label{reducible2}
 Let $p\geq 3$, $r>2p$ and $r\equiv 1\mod(p-1)$, i.e., $b=p$. If $p=3$ and $v(a_p)= \frac{3}{2}$, then further assume that 
 $v(a_p^2-p^3)=3$. If $p^2\mid (p-r)$, then $\bar{V}_{k,a_p}$ is reducible and 
 $$\bar{V}_{k,a_p}\cong \mathrm{unr}\left(\sqrt{-1}\right)\omega\oplus\mathrm{unr}\left(-\sqrt{-1}\right)\omega.$$
\end{theorem}

\begin{proof}
 We  claim that the map $\mathrm{ind}_{KZ}^G(V_{p-2}\otimes D)\twoheadrightarrow\bar{\Theta}_{k,a_p}$ given by Theorem \ref{elimination4} (ii)  factors through 
 $\dfrac{\mathrm{ind}_{KZ}^G(V_{p-2}\otimes D)}{(T^2+1)}\cong\pi\left(p-2,\sqrt{-1},\omega\right)\oplus \pi\left(p-2,-\sqrt{-1},\omega\right)$. 
 Once this claim is proved, the result follows as we know that $\bar{\Theta}_{k,a_p}$ lies in the image of the mod $p$ Local Langlands Correspondence.

 \underline{Proof of the claim}: Let us consider $f=f_2+f_1+f_0\in\mathrm{ind}_{KZ}^G\,\mathrm{Sym}^r\bar{\mathbb{Q}}_p^2$, given by
 \begin{eqnarray*}
  f_2 &=& \underset{\substack{\lambda\in\mathbb{F}_p,\\ \mu\in\mathbb{F}_p^*}}\sum \left[\:g_{2,\,p[\mu]+[\lambda]}^0\,, \frac{1}{a_p}\cdot(Y^r-X^{r-p}Y^{p})\:\right]
  +\underset{\lambda\in\mathbb{F}_p}\sum\left[g_{2,[\lambda]}^0\,,\:\frac{(1-p)}{a_p}\cdot(Y^{r}-X^{r-p}Y^{p})\right],\\
  f_1 &=& \underset{\lambda\in\mathbb{F}_p}\sum\left[\:g_{1,[\lambda]}^0,\:\:\: \frac{(p-1)}{a_p^2}\cdot\sum_{\substack{1<j<r, \\j\equiv 1\mod (p-1)}}\:\alpha_j\cdot\, X^{r-j}Y^j\:\right],\\
  f_0 &=& \left[\mathrm{Id},\: \frac{r}{pa_p}\cdot(X^{r-1}Y-X^{r-p}Y^p)\right],
 \end{eqnarray*}
 where the $\alpha_j$ are integers from Lemma \ref{comb5} (i). If either $p\geq 5$ or $p=3$ and $v(a_p)<3/2$, then we use  the fact  
 that $p^2\mid (p-r)$ 
 and the properties satisfied by the  $\alpha_j$
 to conclude that all of $T^+f_2$, $T^+f_1$, $T^-f_1$, $T^-f_0$, $T^-f_2-a_pf_1+T^+f_0$ are integral and die mod $p$. 
 Thus $(T-a_p)f$ is integral, and is congruent mod $p$ to
 $$ -a_pf_0-a_pf_2=\underset{\substack{\lambda,\, \mu\in\mathbb{F}_p}}\sum \left[\:g_{2,\,p[\mu]+[\lambda]}^0\,,X^{r-p}Y^{p}-Y^r\:\right]
 -\left[\mathrm{Id},\: \dfrac{r}{p}\cdot(X^{r-1}Y-X^{r-p}Y^p)\right].$$ 
 Since $Y^r$, $X^{r-1}Y$ map to $0$ in $Q$, and as $r/p\equiv 1\mod p$ by hypothesis, the image of the  integral function above in $\mathrm{ind}_{KZ}^G Q$ 
 is the same as that of 
 $$-\underset{\lambda,\, \mu\in\mathbb{F}_p}\sum \left[\:g_{2,\,p[\mu]+[\lambda]}^0,\,\theta X^{r-p-1}\right]-\left[\mathrm{Id},\:\theta X^{r-p-1}\right],$$
 which, by the formula for $T^2$ and by Lemma \ref{easylemma}, is the image of 
 $$(T^2+1)[\mathrm{Id},-X^{p-2}]= \underset{\lambda,\, \mu\in\mathbb{F}_p}\sum \left[\:g_{2,\,p[\mu]+[\lambda]}^0,\,-X^{p-2}\right]+\left[\mathrm{Id},\:-X^{p-2}\right]\in\mathrm{ind}_{KZ}^G J_0 = \mathrm{ind}_{KZ}^G(V_{p-2}\otimes D)$$
 in $\mathrm{ind}_{KZ}^G Q$. 
 As the element $[\mathrm{Id}, -X^{p-2}]$ generates the  $G$-module $\mathrm{ind}_{KZ}^G J_0$, the image $(T^2+1)(\mathrm{ind}_{KZ}^G J_0)$ 
 must map to $0$ under the $G$-map $\mathrm{ind}_{KZ}^G J_0\twoheadrightarrow\bar{\Theta}_{k,a_p}$. 

 When $p=3$ and $v(a_p)\geq 3/2$, then $(T-a_p)f$ is not necessarily  integral for the function $f$ above. However, 
 if we consider $f^\prime:=({a_p^2}/{p^3})\cdot f$,  
 then $(T-a_p)f^\prime$ is integral with reduction equal to the image of $c\cdot (T^2+1)[\mathrm{Id}, X]\in\mathrm{ind}_{KZ}^G J_0$ inside  $\mathrm{ind}_{KZ}^G Q$, 
 for $c=\overline{1-a_p^2/p^3}\in\bar \F_p$. By the extra hypothesis when $v(a_p)=3/2$, we see $c$ is  non-zero, and  the result follows as before. 
\end{proof}

\end{document}